\documentclass[10 pt]{article}
\usepackage{amsthm,amsmath,amssymb,amsfonts}
\usepackage[english]{babel}
\usepackage{graphicx}
\usepackage{subfigure}
\usepackage[round]{natbib}
\usepackage{bbm}
\usepackage{enumitem}
\usepackage{color}
\usepackage[T1]{fontenc} 
\usepackage[utf8]{inputenc}
\usepackage[a4paper,left=3.8 cm,right=3.8cm, top=4.2cm, bottom=4.2cm]{geometry}

\def\title#1{{\Large\bf  \begin{center} #1 \vspace{0pt} \end{center}  } \smallskip}
\def\authors#1{{\bf \begin{center} #1 \vspace{0pt} \end{center} } \smallskip}
\def\institution#1{{\sl \begin{center} #1 \vspace{0pt} \end{center} } }

\def\keywords#1{\bigskip \par\noindent{\bf Keywords: }#1\par}
\def\AMS#1{\par\noindent{\bf AMS subject classifications: }#1\par}

\newtheorem{theorem}{Theorem}
\newtheorem{lemma}[theorem]{Lemma}

\newtheorem{cor}[theorem]{Corollary}

\newtheorem{appxlem}{Lemma}[section]

\newtheorem{appxcor}[appxlem]{Corollary}
\newtheorem*{remark*}{Remark}

\newcommand{\eps}{\varepsilon}
\newcommand{\ed}{\exp(\delta)}
\newcommand{\emd}{\exp(-\delta)}
\newcommand{\lrr}[1]{\left( #1 \right)}
\newcommand{\lrs}[1]{\left[ #1 \right]}
\newcommand{\lrc}[1]{\left \{ #1 \right \}}
\newcommand{\lrv}[1]{\left \vert #1 \right \vert}
\newcommand{\argmin}[1]{\underset{#1}{\operatorname{arg} \operatorname{min}}}
\newcommand{\Mn}{\mathcal{M}_n}
\newcommand{\abs}[1]{\vert #1 \vert}

\newcommand{\R}{{\mathbb R}}
\newcommand{\E}{{ \mathbb E}}
\renewcommand{\P}{{\mathbb P}}
\newcommand{\N}{{\mathbb N}}

\renewcommand{\L}{{\mathbb L}}
\newcommand{\e}{\text{e}}

\DeclareMathOperator{\Var}{Var}

\DeclareMathOperator{\trace}{trace}

\makeatletter
\newcommand{\vast}{\bBigg@{3}}
\newcommand{\Vast}{\bBigg@{5}}
\makeatother

\begin{document}
\sloppy

\title{Prediction out-of-sample using block shrinkage estimators: model selection and predictive inference}

\authors{Hannes Leeb and Nina S.\ Senitschnig} 
\institution{Department of Statistics, University of Vienna}

\bigskip
\noindent {\Large\bf Abstract}
\medskip

In a linear regression model with random design, we consider a family of candidate models from which we want to select a `good' model for prediction out-of-sample. We fit the models using block shrinkage estimators, and we focus on the challenging situation where the number of explanatory variables can be of the same order as sample size and where the number of candidate models can be much larger than sample size. We develop an estimator for the out-of-sample predictive performance, and we show that the empirically best model is asymptotically as good as the truly best model. Using the estimator corresponding to the empirically best model, we construct a prediction interval that is approximately valid and short with high probability, i.e., we show that the actual coverage probability is close to the nominal one and that the length of this prediction interval is close to the length of the shortest but infeasible prediction interval. All results hold uniformly over a large class of data-generating processes. These findings extend results of Leeb (2009, Ann. Stat. 37:2838-2876), where the models are fit using least-squares estimators, and of \cite{huber13}, where the models are fit using shrinkage estimators without block structure.

\keywords{(block) James--Stein estimator, prediction out-of-sample, random design}
\AMS{62M20, 62J07}  

\newpage

\section{Introduction}
\label{introduction}

We study the problem of how to find a `good' model for prediction out-of-sample and how to use this model for designing `valid' and `short' prediction intervals. We assume that the data are generated from a possibly infinite dimensional linear model where we impose some technical but rather innocuous assumptions. In contrast to other authors, we do not impose sparsity on the parameters, and we do not require a special structure of the covariance matrix. We allow the collection of models to be very large and the models in our collection to be very complex, i.e., we allow the number of models to exceed sample size and the number of parameters to be of the same order as sample size. The models are fit using a block James--Stein estimator. Our results are conditional on a training sample when we repeatedly predict future observations, i.e., in our analysis we fix the training data and average over future observations.  All the results presented are finite sample results. 

Model selection is a well studied field in statistics. It is also well known that inference after model selection needs special treatment, i.e., ignoring the selection step and doing inference as if the model was chosen a priori leads to invalid conclusions because model selection is usually data-driven and hence random (see e.g.\ \cite{poetscher91} or \cite{leebpoetscher05}). For an overview of model selection procedures and the properties of post-selection estimators see \cite{leebpoetscher08}. 

\cite{berketal13} propose valid confidence intervals post-model selection regardless of which model selection procedure was used. \cite{bachocleebpoetscher18} generalize these results to post-model selection predictors. Both articles consider linear models with homoscedastic errors. \cite{bachocdavidlukas17} develop a general framework allowing for linear models with heteroscedastic errors or binary response models with general link functions. The results in these papers can not be compared to ours because they are covering a non-standard and model dependent target. \cite{leeetal15}, \cite{leeetal16} and \cite{tibshiranietal16} (and the references in these papers) focus on the same target and discuss post-model selection inference in a `condition on selection' framework. Post-model selection inference on the conventional parameters is covered in \cite{poetscherschneider10} and \cite{schneider16}. They consider confidence sets based on thresholding estimators in Gaussian linear regression models. The papers by \cite{bellonietal13} and \cite{bellonietal14} as well as by \cite{vandegeer14} and \cite{zhangzhang14} develop valid inference procedures under sparsity conditions. It is crucial to emphasize that we do not assume sparsity or any special structure of the unknown parameters. With regard to form and content, this work is closely related to the paper by \cite{leeb09} where the models are fit using least-squares estimators. Here, we extend these results to a larger class of estimators.

This paper is organized as follows: we give an introduction to the setting and the framework in Section~\ref{sec:framework}, i.e., we describe the overall model, the collection of candidate models and how we fit the candidate models using block James--Stein-type shrinkage estimators. Furthermore, we introduce how we measure the performance of the competing models via the conditional mean-squared prediction error. In Section~\ref{sec:selection}, we deal with selecting the best model and show how we can estimate its performance. Section~\ref{sec:inference} addresses the problem of constructing prediction intervals and shows that the intervals are asymptotically valid and short with high probability.

\section{Framework}
\label{sec:framework}

As data-generating process, we consider a linear regression model that can be infinite dimensional. The candidate models correspond to finite dimensional submodels of this overall model. For simplicity, we assume that we do not have an intercept and that the explanatory variables are centered.\footnote{We believe that including an intercept and non-centered explanatory variables does not change the results qualitatively and that we can handle this task using similar methods as in \cite{leeb09}.} We further assume that we have a block structure in the data and that we would like to estimate and shrink the parameters corresponding to the blocks differently.\footnote{For the unblocked case, see \cite{leebsenitschnig15}.} Of course, there are different ways of shrinking in a block design, we will discuss and further describe the one strategy we pursue. We use a training sample to fit the models and measure the out-of sample predictive performance of each model using the conditional mean-squared prediction error when repeatedly predicting future observations keeping the training sample fixed. 

More formally, as data-generating process we consider a response variable $y$, a sequence of stochastic explanatory variables $x=(x_i)_{i \geq 1}$, a sequence of unknown parameters $\beta=(\beta_i)_{i \geq 1}$ and an error term $u$ that are related via
\begin{gather} \label{eq:overall.model}
y = \sum_{i=1}^{\infty} x_i \beta_i + u.
\end{gather}
Throughout the paper, we will assume that the error $u$ is centered with variance $\sigma^2>0$, that $x$ is centered with variance-covariance matrix $\Sigma=\E[x_i x_j]_{i \geq 1, j \geq 1}$ such that the $x_i$’s are not perfectly correlated among themselves, i.e., we require for each $k \geq 1$ and integers $i_1, \ldots, i_k$ that the variance-covariance matrix of $(x_{i_1}, \ldots, x_{i_k})'$ is positive definite, that $u$ and $x$ are independent and that the series converges in squared mean. The assumptions made so far are rather standard and innocuous. We assume further that $(y, x)$ is jointly Gaussian. We heavily rely on Gaussianity because we need the conditional mean to be linear and the conditional variance to be constant, a property that is fulfilled only for the Gaussian distribution. Results of \cite{steinbergerleeb18} and \cite{milovic15} show that approximate linearity of the conditional expectation holds for a large class of distributions. Ongoing work of \citeauthor{milovic15} and \citeauthor{leeb09} deal with the conditional variance being approximately constant. We are confident that it is possible to get rid of the normality assumption using their results.

We are given a sample of size $n$ which will be denoted by $(Y,X)$, where $Y = (y^{(1)}, \ldots, y^{(n)})'$ is a $n$-vector and $X = (x^{(1)}, \ldots, x^{(n)})'$ is a $n \times \infty$ net and where $(y^{(j)}, x^{(j)})$ are i.i.d.\ copies of the random variables in \eqref{eq:overall.model}. Further, we consider a collection of candidate models $\mathcal{M}_n$ that are finite-dimensional submodels of the overall model in \eqref{eq:overall.model}, where we restrict some components of $\beta$ to zero. Each of these models can be identified by a $0$-$1$ sequence $m=(m_i)_{i \geq 1}$, where $m_i=0$ if the corresponding $\beta_i$ is restricted to zero and where $m_i=1$ if $\beta_i$ is not restricted. For every model $m \in \Mn$, we denote by $\abs{m}$ the number of unrestricted components of $\beta$, i.e., $\abs{m}=\sum_{i \geq 1} m_i$, and we assume that $6 \leq \abs{m} <n$.

Throughout, we consider fixed parameters $\beta$, $\sigma^2$ and $\Sigma$ as in \eqref{eq:overall.model}, a fixed sample size $n$ and a fixed model $m$. Let $\hat{\beta}^B(m)$ denote the estimator of $\beta$ in model $m$. If $m_i=0$, then the $i$-th component of $\hat{\beta}^B(m)$ is defined as zero. For the remaining components, note the following: We write $x(m)$ for those entries in $x$ where $m_i=1$, and we write $X(m)$ for those columns of $X$ that are included in the submodel $m$, i.e., $X(m)$ is a $n \times \abs{m}$ matrix. Because  $(Y, X)$ consists of i.i.d.\ samples of the pair $(y,x)$ it follows that $(Y, X(m))$ consists of i.i.d.\ samples of the pair $(y, x(m))$. Because we assumed that $(y,x)$ is jointly Gaussian, we know that the conditional distribution of $y$ given $x(m)$ is Gaussian where the conditional mean is linear in $x(m)$, i.e., equals $x(m)'\theta$ for some appropriate $\abs{m}$-vector $\theta$,  and the conditional variance is constant in $x(m)$ and equals, say, $\sigma^2(m)$. Hence, we can write
\begin{gather} \label{eq:model.m}
Y= X(m) \theta + w
\end{gather}
where $w \sim N(0, \sigma^2(m) I_n)$ for some $\sigma^2(m)>0$ and where $X(m)$ and $w$ are independent. We assume that $X(m)$ consists of two blocks of dimension $n \times \abs{m_1}$ and $n \times \abs{m_2}$ with $\abs{m}=\abs{m_1} + \abs{m_2}$ such that $3 \leq \abs{m_1}$ and $3 \leq \abs{m_2}$, i.e., $X(m)=(X_1(m), X_2(m))$. We can rewrite the model in the preceding display as
\begin{align}
Y & = X_1(m) \theta_1 + X_2(m) \theta_2 + w \nonumber \\
& = X_1(m) \theta_1^* + X_2^*(m) \theta_2 + w, \label{eq:model.o}
\end{align}
where $\theta_1^* = \theta_1 + (X_1(m)'X_1(m))^{-1} X_1(m)'X_2(m) \theta_2$ and $X_2^*(m) = M_1(m) X_2(m)$ with $M_1(m)= I_n - X_1(m)(X_1(m)'X_1(m))^{-1}X_1(m)'$, i.e., the projection on the orthogonal complement of the column span of $X_1(m)$. On the probability zero event where the inverse matrix does not exist, we use the Moore-Penrose inverse instead of the usual inverse. Note that the two regressors in \eqref{eq:model.o} are orthogonal so that we can estimate $\theta_1^*$ and $\theta_2$ separately. Let $\hat{\theta}_1^*$ and $\hat{\theta}_2$ be the least-squares estimators of $\theta_1^*$ and $\theta_2$, i.e., $\hat{\theta}_1^* = (X_1(m)'X_1(m))^{-1}X_1(m)'Y$ and $\hat{\theta}_2=(X_2(m)'M_1(m)X_2(m))^{-1}X_2(m)'M_1(m) Y$. Let $\hat{\theta}_{1}^{*JS}$ and $\hat{\theta}_{2}^{JS}$ be the positive part James--Stein-type shrinkage estimators that are obtained by shrinking the least-squares estimators $\hat{\theta}_1^*$ and $\hat{\theta}_2$, i.e., %
\begin{align}
\hat{\theta}_1^{*JS} & =  \lrr{ 1 - c_1 \hat{\sigma}^2(m) \frac{ \abs{m_1}}{\hat{\theta}_1^{*}{'}X_1(m)'X_1(m)\hat{\theta}_1^* } }_+ \hat{\theta}_1^{*},  \label{eq:theta1s.js} \\
\hat{\theta}_2^{JS} & = \lrr{1- c_2 \hat{\sigma}^2(m) \frac{\abs{m_2}}{\hat{\theta}_2'X_2^*(m)'X_2^*(m) \hat{\theta}_2} }_+ \hat{\theta}_2,  \label{eq:theta2.js}
\end{align}
where $(x)_+=\max \{x,0 \}$, where $c_1 \geq 0$ and $c_2 \geq 0$ are tuning parameters and where $\hat{\sigma}^2(m)$ is the usual unbiased variance estimator in model \eqref{eq:model.m}. Note that we can rewrite the estimators as $\hat{\theta}_1^{*JS}  =  \lrr{ 1 - a_1(m) } \hat{\theta}_1^{*} $ and $\hat{\theta}_2^{JS} = \lrr{1- a_2(m) } \hat{\theta}_2$ where the shrinkage factors are
 \begin{align}
 \begin{split}  \label{eq:a1a2}
 a_1(m) & =\min \lrc{ 1, c_1 \hat{\sigma}^2(m) \frac{ \abs{m_1}}{\hat{\theta}_1^{*}{'}X_1(m)'X_1(m)\hat{\theta}_1^*}}, \\
 a_2(m) & = \min \lrc{1, c_2 \hat{\sigma}^2(m) \frac{\abs{m_2}}{\hat{\theta}_2'X_2^*(m)'X_2^*(m) \hat{\theta}_2} }. 
 \end{split}
\end{align}
Because of the definition of $\theta_1^*$, we set $\hat{\theta}_1^{JS} = \hat{\theta}_1^{*JS} - (X_1(m)'X_1(m))^{-1} X_1(m)'X_2(m) \hat{\theta}_2^{JS}$. Hence, for the remaining $\abs{m}$ components of $\hat{\beta}^B(m)$, we use the vector $(\hat{\theta}_1^{JS}{'}, \hat{\theta}_2^{JS}{'})'$.

We assume that we have a new copy of the random variables $(y,x)$, independent of the training sample, that we will denote by $(y^{(0)}, x^{(0)})$. We predict $y^{(0)}$ using the predictor
\begin{gather*}
\hat{y}^{(0)}(m) = \sum_{i=1}^{\infty} x_i^{(0)} \hat{\beta}_i^B(m).
\end{gather*}
The conditional mean-squared prediction error corresponding to model $m$ will be denoted by $\rho^2(m)$ and is defined as
\begin{gather*}
\rho^2(m) = \E \lrs{ \lrr{ \hat{y}^{(0)}(m) - y^{(0)} }^2 \Big \Vert X, Y},
\end{gather*}
where the expectation is taken with respect to $(y^{(0)}, x^{(0)})$ and where the training sample is treated as fixed. We are interested in the model that performs best within our class of candidate models, i.e., we are looking for the minimizer of $\rho^2(m)$ over $m \in \mathcal{M}_n$. Because $\rho^2(m)$ depends on the unknown parameters $\beta$, $\Sigma$ and $\sigma^2$ in a complicated fashion, we approximate it by an empirical counterpart that is defined as follows
\begin{align*}
\hat{\rho}^2(m) & = w_1 \hat{\sigma}^2(m)  + w_2  \frac{Y' (I_n - M_1(m)) Y}{n} + w_3 \frac{Y' M_1(m) Y}{n-\abs{m_1}},
\end{align*}
where the weights equal
\begin{align*}
w_1 &=  (1-a_2(m))^2 \frac{\abs{m}}{n-\abs{m}+1}  +1 -a_2(m)^2 - (a_1(m)- a_2(m))^2 \frac{\abs{m_1}}{n-\abs{m_1}+1} \\
& \phantom{= } +(a_2(m) - a_1(m))(2 - a_1(m) - a_2(m))  \frac{\abs{m_1}}{n-\abs{m_1}+1},   \\
w_2 & =  a_1(m)^2, \\
w_3 & =   a_2(m)^2 - a_1(m)^2 \frac{\abs{m_1}}{n} + (a_2(m)- a_1(m))^2 \frac{\abs{m_1}}{n-\abs{m_1}+1}.
\end{align*}
%
In Appendix \ref{sec:pe}, we have a closer look on the derivation of $\hat{\rho}^2(m)$. The following result shows that for every model $m$ the empirical mean-squared prediction error $\hat{\rho}^2(m)$ is a good approximation for the true mean-squared prediction error $\rho^2(m)$.
\begin{theorem} \label{th:help.bound}
For every fixed $m \in \mathcal{M}_n$, we have for every $\eps > 0$
\begin{align} 
\begin{split} \label{eq:help.bound}
\P & \lrr{ \lrv{ \log  \frac{\hat{\rho}^2(m)}{\rho^2(m)}   } \geq \eps } \leq 31 \abs{m} \exp \lrr{ - n \lrr{\frac{\abs{m_1}}{n}}^2 \lrr{1 - \frac{\abs{m}}{n} }^5   \frac{ \eps^2}{ 14397 (1+\eps)^2 } },
\end{split}
\end{align}
where $\rho^2(m)$ and $\hat{\rho}^2(m)$ are defined above. Alternatively, we can bound the left-hand side in the preceding inequality from above by
\begin{align} \label{eq:help.bound.m}
31 \abs{m} \exp \lrr{ - n \lrr{1 - \frac{\abs{m}}{n} }^5   \frac{\eps^2}{28279 (1+\mu(m))^2 (1+ \eps)^2 } }.
\end{align}
where $\mu(m) = \theta' \Sigma(m) \theta/\sigma^2(m)$ and where $\Sigma(m)$ is the variance-covariance matrix of $x(m)$.
\end{theorem}
This result shows that for any fixed model the true and the empirical mean-squared prediction error are close to each other. 
Noting that $\rho^2(m) \geq \sigma^2 > 0$, the term $\hat{\rho}^2(m)/\rho^2(m)$ is always well-defined. On the probability zero event where $\hat{\rho}^2(m)=0$, $\log ( \hat{\rho}^2(m)/\rho^2(m))$ should be interpreted as $\infty$. There are two different upper bounds in the previous result. Both upper bounds depend on known quantities like $\eps$, $n$, $\abs{m}$ and $\abs{m_1}$. The upper bound in \eqref{eq:help.bound} does not depend on unknown quantities whereas the upper bound in \eqref{eq:help.bound.m} depends on the unknown signal-to-noise ratio $\mu(m)$. For every fixed model $m$, we can estimate the signal to noise ratio $\mu(m)$ by $(Y'Y/n)/\hat{\sigma}^2(m)-1$. It should be noted that both upper bounds do not tend to zero as $\eps$ gets larger. We could present a smaller upper bound but at the expense of a more complicated and complex structure of the bound that does not clearly show the effect of the individual quantities. 

\section{Model selection}
\label{sec:selection}

In this section, we use Bonferroni's inequality to extent Theorem~\ref{th:help.bound} to hold uniformly over the whole class of candidate models $\Mn$ and we show how to use the result for model selection. We will not assume that one of the candidate models is the true model (because it is not the aim of the paper to find the true model). Rather we would like to find a `good' model for prediction out-of-sample, that is a model having a small mean-squared prediction error. Because minimizing $\rho^2(m)$ is unfeasible, we minimize the empirical mean-squared prediction error $\hat{\rho}^2(m)$ instead. Lemma~\ref{lem:bound.uniform} and the subsequent corollary motivate this approach.
\begin{lemma} \label{lem:bound.uniform}
Consider a finite and non-empty collection of candidate models $\Mn$ and let $r_n = \inf_{m \in \Mn} \abs{m_1}$ and $s_n = \sup_{m \in \Mn} \abs{m}$. Then, we have for each $\eps >0$
\begin{align}
\begin{split} \label{eq:bound.uniform} 
\P & \lrr{ \sup_{m \in \Mn} \lrv{ \log  \frac{\hat{\rho}^2(m)}{\rho^2(m)}   } \geq \eps }  \\
 & \leq
31 \abs{\Mn} s_n \exp \lrr{- n \lrr{\frac{r_n}{n}}^2 \lrr{ 1 - \frac{s_n}{n}}^5 \frac{\eps^2}{14397(1+\eps)^2} },
\end{split}
\end{align}
where $\abs{\Mn}$ denotes the number of candidate models in collection $\Mn$. The result holds uniformly over all data-generating processes as in \eqref{eq:overall.model}.
Alternatively, we can bound the left-hand side in the preceding display from above by
\begin{align} \label{eq:bound.uniform.m}
31 \abs{\Mn} s_n \exp \lrr{- n \lrr{ 1 - \frac{s_n}{n}}^5 \frac{\eps^2}{28279 (1+\eps)^2 d^2} },
\end{align}
where the result holds uniformly over all data-generating processes as in \eqref{eq:overall.model} such that $\Var(y)/\sigma^2 \leq d$ for some $d >0$.
\end{lemma}
We have two different types of exchangeable upper bounds. The upper bound in \eqref{eq:bound.uniform} only depends on known quantities and is exponentially small in $n$ if only $r_n/n$ is not too small and if $s_n/n$ and $\abs{\Mn}$ are not too large. So we need the number of unrestricted components in the first block to be large, more precisely, we need $r_n/n$ to be bounded away from 0, i.e., $r_n/n > \eta_1$ for some $\eta_1>0$ and for all $n \in \N$. The models can also not be too complex, i.e., $s_n/n$ should be bounded away from 1, i.e., $s_n/n < 1-\eta_2$ for some $\eta_2>0$ and for all $n \in \N$. Furthermore, we see that the number of models in collection $\Mn$ can exceed sample size (and can actually be a large multiple of sample size) but it can not be too large, e.g., complete subset selection is not possible. The upper bound in \eqref{eq:bound.uniform.m} also depends on the unknown quantity $d$ which is an upper bound on the signal-to-noise ratio of the data-generating process and is exponentially small in $n$ if only $s_n/n$, $\abs{\Mn}$ and $d$ are not too large. Note that the factor $s_n$ outside of the exponential term in both bounds is negligible.  

A simple consequence of the preceding result is that the empirically best model is a `good' model. 
For this purpose, let $\hat{m}_n^*$ and $m_n^*$ be minimizers of $\hat{\rho}^2(m)$ and of $\rho^2(m)$, respectively, i.e., 
\begin{gather*}
\hat{m}_n^* = \argmin {m \in \Mn} \hat{\rho}^2(m), \quad m_n^* = \argmin{m \in \Mn} \rho^2(m).
\end{gather*}
\begin{cor} \label{cor:bound.best}
Let $r_n = \inf_{m \in \Mn} \abs{m_1}$ and $s_n = \sup_{m \in \Mn} \abs{m}$ Then, we have for each $\eps >0$
\begin{align}
\begin{split} \label{eq:true.perf}
\P & \lrr{ \lrv{ \log  \frac{\rho^2(\hat{m}_n^*)}{\rho^2(m_n^*)}   } \geq \eps }  \leq 31 \abs{\Mn} s_n \exp \lrr{-n  \lrr{ \frac{r_n}{n}}^2 \lrr{ 1 - \frac{s_n}{n} }^5 \frac{\eps^2}{14397(2+\eps)^2} },
\end{split}
\intertext{as well as}
\begin{split} \label{eq:est.perf}
\P & \lrr{ \lrv{ \log  \frac{\hat{\rho}^2(\hat{m}_n^*)}{\rho^2(\hat{m}_n^*)}   } \geq \eps }  \leq 31 \abs{\Mn} s_n \exp \lrr{-n  \lrr{ \frac{r_n}{n}}^2 \lrr{ 1 - \frac{s_n}{n} }^5 \frac{\eps^2}{14397(1+\eps)^2} },
\end{split}
\end{align}
where $\Mn$ denotes the number of candidate models in collection $\Mn$. Both results hold uniformly over all data-generating processes as in \eqref{eq:overall.model}. Alternatively, we can bound the left-hand side in \eqref{eq:true.perf} from above by
\begin{align} \label{eq:true.perf.m}
31 \abs{\Mn} s_n \exp \lrr{-n  \lrr{1 - \frac{s_n}{n}}^5 \frac{\eps^2}{28279(2+\eps)^2d^2}   },
\intertext{and the left-hand side in \eqref{eq:est.perf} from above by}
\label{eq:est.perf.m}
31 \abs{\Mn} s_n \exp \lrr{-n  \lrr{1 - \frac{s_n}{n}}^5 \frac{\eps^2}{28279(1+\eps)^2d^2}   }.
\end{align}
Both results hold uniformly over all data-generating processes as in \eqref{eq:overall.model} such that $\Var(y)/\sigma^2 \leq d$ for some $d>0$.
\end{cor}
The results in \eqref{eq:true.perf} and \eqref{eq:true.perf.m} show that the empirically best model is asymptotically as good as the truly best model, in the sense that the true mean-squared prediction error of the truly best model lies close to the true mean-squared prediction error of the empirically best model. The results in \eqref{eq:est.perf} and \eqref{eq:est.perf.m} show that the empirical mean-squared prediction error of the empirically best model lies close to its true mean-squared prediction error. This implies that we can estimate the true performance of the empirically best model just by plugging it into the empirical mean-squared prediction error.

\section{Statistical inference}
\label{sec:inference}

In this section, we construct prediction intervals and we show that these intervals have the desired properties of being asymptotically `valid' and asymptotically `short'.

For a fixed model $m$, the prediction error equals $y^{(0)} - \hat{y}^{(0)}(m)$. Conditional on the training sample this prediction error follows a centered normal distribution with variance $\rho^2(m)$. We will denote this distribution by $\mathbb{L}(m)$, i.e., $\mathbb{L}(m) \equiv N(0, \rho^2(m))$. Using this distribution to construct prediction intervals for $y^{(0)}$ is infeasible because it depends on unknown quantities via its variance $\rho^2(m)$. Let $\widehat{\mathbb{L}}(m) \equiv N(0, \hat{\rho}^2(m))$ be an approximation to the true distribution and use this distribution to construct the prediction interval. The next result shows that $\mathbb{L}(m)$ is close to $\widehat{\mathbb{L}}(m)$ in the sense that their total variation distance is small with high probability.

\begin{theorem} \label{th:bound.tv}
For a fixed model $m \in \mathcal{M}_n$, we have for all $\eps > 0$
\begin{align} \label{eq:bound.tv}
\begin{split}
\P &\lrr{ \Vert \mathbb{L}(m) - \widehat{\mathbb{L}}(m) \Vert_{TV}  \geq \eps} \\
& \leq  31 \abs{m} \exp \lrr{- n \lrr{\frac{\abs{m_1}}{n}}^2 \lrr{1-\frac{\lrv{m}}{n}}^5 \frac{\eps^2}{900(1+4 \eps)^2} }.
\end{split}
\intertext{Alternatively, we can bound the left-hand side in the preceding display from above by}
\label{eq:bound.tv.m}
& 31 \abs{m} \exp \lrr{- n \lrr{1-\frac{\lrv{m}}{n}}^5 \frac{ \eps^2}{1768(1+4 \eps)^2 (1+\mu(m))^2} }.
\end{align}
\end{theorem}

This result together with Bonferroni's inequality gives a uniform result over the whole class of candidate models. 

\begin{cor} \label{cor:bound.tv.uniform}
Let $r_n = \inf_{m \in \Mn} \abs{m_1}$ and $s_n = \sup_{m \in \Mn} \abs{m}$. Then, we have for each $\eps >0$
\begin{align} \label{eq:bound.tv.uniform}
\begin{split}
\P &\lrr{ \sup_{m \in \Mn} \Vert \mathbb{L}(m) - \widehat{\mathbb{L}}(m) \Vert_{TV}  \geq \eps} \\
 & \leq  31 \abs{\mathcal{M}_n} s_n \exp \lrr{- n \lrr{\frac{r_n}{n}}^2 \lrr{1-\frac{s_n}{n}}^5 \frac{\eps^2}{900(1+4 \eps)^2} }.
\end{split}
\intertext{The result holds uniformly over all data generating processes as in \eqref{eq:overall.model}. Alternatively, we can bound the left-hand side in the preceding display from above by}
\label{eq:bound.tv.m}
& \phantom{\leq e}31 \abs{\mathcal{M}_n} s_n \exp \lrr{- n \lrr{1-\frac{s_n}{n}}^5 \frac{ \eps^2}{1768(1+4 \eps)^2 d^2} }.
\end{align}
The result holds uniformly over all data generating processes as in \eqref{eq:overall.model} such that $\Var(y)/\sigma^2 \leq d$ for some $d >0$.
\end{cor}

Because $y^{(0)} - \hat{y}^{(0)}(m)$ is distributed as $\mathbb{L}(m)$, we see that an infeasible prediction interval for $y^{(0)}$ with coverage probability $1 - \alpha$ for some $\alpha \in (0,1)$ is given by 
\begin{align*} 
\hat{y}^{(0)}(m) \pm Q_{1-\alpha/2} \rho(m),
\end{align*}
where $Q_{1-\alpha/2}$ is the $1-\alpha/2$ quantile of the standard normal distribution. The length of this `prediction interval' equals $2 Q_{1-\alpha/2} \rho(m)$ and is minimal for $m_n^*$ with $2 Q_{1-\alpha/2} \rho(m_n^*)$. Using $\hat{\rho}(m)$ instead of $\rho(m)$ in the previous display, we define the prediction interval as
\begin{align*}
\mathcal{I}(m): \hat{y}^{(0)}( m ) \pm Q_{1-\alpha/2} \hat{\rho}( m).
\end{align*}
The length of this interval is minimized for $\hat{m}_n^*$. The next results shows that the coverage probability of $\mathcal{I}(\hat{m}_n^*)$, conditional on the training sample, is close to the nominal one, except on an event that has probability converging to zero as $n$ increases under the same conditions we had in the previous section.

\begin{cor} \label{cor:pi.valid}
For every $\eps > 0$, we have that 
\begin{gather} \label{eq:pi.valid}
\lrv{(1- \alpha) - \P (y^{(0)} \in \mathcal{I}(\hat{m}_n^*)) } \leq \eps
\end{gather}
except on an event whose probability is not larger than
\begin{align*}
& 31 \abs{\mathcal{M}_n} s_n \exp \lrr{- n \lrr{\frac{r_n}{n}}^2 \lrr{1-\frac{s_n}{n}}^5 \frac{\eps^2}{900(1+4 \eps)^2} }
 \intertext{uniformly over all data generating processes as in \eqref{eq:overall.model}. Alternatively, we can bound the probability of the exception event from above by}
 & 31 \abs{\mathcal{M}_n} s_n \exp \lrr{- n \lrr{1-\frac{s_n}{n}}^5 \frac{ \eps^2}{1768(1+4 \eps)^2 d^2} }
\end{align*}
uniformly over all data generating processes as in \eqref{eq:overall.model} such that $\Var(y)/\sigma^2 \leq d$ for some $d>0$.
\end{cor}
The next result shows that the minimal length of the infeasbile prediction interval lies close to the length of $\mathcal{I}(\hat{m}_n^*)$.
\begin{cor} \label{cor:pi.short}
For each $\eps > 0$, we have
\begin{align} \label{eq:pi.short}
\begin{split}
\P \lrr{ \lrv{ \log \frac{\hat{\rho}(\hat{m}_n^*)}{ \rho( m_n^*) } } \geq \eps } & \leq 31 \abs{\mathcal{M}_n} s_n \exp \lrr{- n \lrr{\frac{r_n}{n}}^2 \lrr{1-\frac{s_n}{n}}^5 \frac{\eps^2}{3600(1+2 \eps)^2} }
\end{split}
\intertext{uniformly over all data generating processes as in \eqref{eq:overall.model}. Alternatively, we can bound the left-hand side in \eqref{eq:pi.short} from above by}
\label{eq:pi.short.m}
 & \phantom{\leq } 31 \abs{\mathcal{M}_n} s_n \exp \lrr{- n \lrr{1-\frac{s_n}{n}}^5 \frac{ \eps^2}{7070(1+ 2 \eps)^2 d^2} }
\end{align}
uniformly over all data generating processes as in \eqref{eq:overall.model} such that $\Var(y)/\sigma^2 \leq d$ for some $d>0$.
\end{cor}

\section{Conclusion}
\label{sec:conclusion}

We have shown how to select a model that performs well for prediction out-of-sample and how to use this model to construct prediction-intervals. We measured the performance of a model by its mean-squared prediction error conditional on the training data $(Y,X)$ when repeatedly predicting over future observations, i.e., in our analysis we kept the training data fixed and averaged over future observations. Fitting the models using least-squares estimators was done in \cite{leeb08} and \cite{leeb09}  and fitting them using usual James--Stein-type shrinkage estimators was done in \cite{huber13}. In this work, we have considered a larger class of estimators namely blocked James-Stein-type shrinkage estimators. The models can be very complex in the sense that the number of parameters can grow with sample size but can never exceed sample size, and the collection of models can be huge, it can be much larger than sample size.

Because the true out-of-sample prediction error $\rho^2(m)$ is not known, we minimize its empirical counterpart $\hat{\rho}^2(m)$. First of all, we have shown that those two quantities lie close to each other in the sense that $\P( \abs{ \log (\hat{\rho}^2(m)/\rho^2(m)) } \geq \eps )$ is bounded from above by a bound that is exponentially small in $n$ under some restrictions (see \eqref{eq:help.bound} and \eqref{eq:help.bound.m} and the discussion following that result). We have shown that the true performance of the model $\hat{m}^*_n$ that minimizes $\hat{\rho}^2(m)$ lies close to the minimal true performance (see \eqref{eq:true.perf} and \eqref{eq:true.perf.m}) and that we can use its empirical mean-squared prediction error to estimate its true performance (see \eqref{eq:est.perf} and \eqref{eq:est.perf.m}).
Designing prediction intervals, we have used again the empirical counterpart $\hat{\rho}^2(m)$ instead of $\rho^2(m)$. We have shown that the interval $\mathcal{I}(\hat{m}^*_n)$ has actual coverage probability that lies close to the nominal one except on an event that has a probability converging to 0 exponentially fast in $n$ (see Corollary~\ref{cor:pi.valid}). Furthermore, the length of this prediction interval is short in the sense that it is close to the length of an infeasible prediction interval that depends on the unknown out-of-sample prediction error $\rho^2(m)$ (see \eqref{eq:pi.short} and \eqref{eq:pi.short.m}).
It should be noted that all our results are finite sample results that hold for every sample size $n$, and that hold uniformly over a large class of data-generating processes.

\newpage
\begin{appendix}

\section{Technical details for deriving $\hat{\rho}^2(m)$}
\label{sec:pe}

To shorten the notation in the appendix, we drop the dependence on $m$ in the notation and we keep in mind that we always consider a fixed model $m$. Thus, let $X(m)=Z$, $X_1(m)=Z_1$, $X_2(m)=Z_2$, $\sigma^2(m)=s^2$ and $\Sigma(m)=\E[x(m)x(m)']= S$. The model in \eqref{eq:model.m} then becomes
\begin{gather} \label{eq:model.z}
Y=Z \theta + w,
\end{gather}
where $w \sim N(0, s^2 I_n)$ is independent of $Z$. Because $Z$ is divided into two blocks, we also rewrite $\theta$ as $\theta=(\theta_1', \theta_2')'$ with $\theta_1$ and $\theta_2$ being a $\abs{m_1}$-vector and a $\abs{m_2}$-vector, respectively, and $S$ as 
\begin{gather*}
S = 
\begin{pmatrix}
S_{1,1} & S_{1,2}\\
S_{2,1} & S_{2,2}
\end{pmatrix},
\end{gather*}
where $S_{1,1}$ is $\abs{m_1} \times \abs{m_1}$, $S_{1,2}$ is $\abs{m_1} \times \abs{m_2}$ and $S_{2,2}$ is $\abs{m_2} \times \abs{m_2}$. Note that $S_{2,1} = S_{1,2}'$ because $S$ is symmetric. For motivating the formula for $\hat{\rho}^2(m)$, note that the true mean-squared prediction error equals
\begin{gather*}
\rho^2(m)  = (\hat{\theta}^{BJS} - \theta )'S(\hat{\theta}^{BJS} - \theta ) + s^2,
\end{gather*}
where $\hat{\theta}^{BJS}$ is the blocked James--Stein-estimator for $\theta$ as defined in Section \ref{sec:framework}, i.e., 
\begin{gather*}
\hat{\theta}^{BJS} = 
\begin{pmatrix}
\hat{\theta}_1^{JS} \\ \hat{\theta}_2^{JS}
\end{pmatrix} =
\begin{pmatrix}
\hat{\theta}_1^{*JS} - (Z_1'Z_1)^{-1}Z_1'Z_2 \hat{\theta}_2^{JS} \\ \hat{\theta}_2^{JS}
\end{pmatrix},
\end{gather*}
where $\hat{\theta}_1^{*JS} = (1 - a_1 ) \hat{\theta}_1^{*}$ and $\hat{\theta}_2^{JS} = (1-a_2) \hat{\theta}_2$ with\footnote{Umbenennen $a_1 = a_n^{(1)}$ und $a_2 = a_n^{(2)}$}
\begin{align} \label{eq:a1a2}
a_1 & =  \min \lrc{ c_1 \hat{s}^2 \frac{\abs{m_1}}{\hat{\theta}_1^{*}{'}Z_1'Z_1\hat{\theta}_1^* }, 1 }, \; a_2  = \min \lrc{ c_2 \hat{s}^2 \frac{\abs{m_2}}{\hat{\theta}_2'Z_2'M_1Z_2 \hat{\theta}_2}, 1},
\end{align}
where $\hat{s}^2$ is the usual unbiased variance estimator for $s^2$ in model \eqref{eq:model.z} (of course $\hat{s}^2=\hat{\sigma}^2(m)$, $a_1=a_1(m)$ and $a_2=a_2(m)$ as in Section \ref{sec:framework}), where $\hat{\theta}_1=(Z_1'Z_1)^{-1} Z_1'Y$ and $\hat{\theta}_2=(Z_2'M_1 Z_2)^{-1} Z_2'M_1 Y$ with $M_1=I_n - Z_1(Z_1'Z_1)^{-1}Z_1'$. Let $\hat{\theta} = (\hat{\theta}_1', \hat{\theta}_2')'$ be the least-squares estimator in model \eqref{eq:model.z} and note that $\hat{\theta}_1=\hat{\theta}_1^* - (Z_1'Z_1)^{-1} Z_1'Z_2 \hat{\theta}_2$.
Rewriting the James--Stein-type shrinkage estimator for $\theta_1$ as
\begin{align*}
\hat{\theta}_1^{JS} & =  \hat{\theta}_1^* - a_1 \hat{\theta}_1^* - (Z_1'Z_1)^{-1} Z_1'Z_2 \hat{\theta}_2 + a_2 (Z_1'Z_1)^{-1} Z_1'Z_2  \hat{\theta}_2 \\
& = \hat{\theta}_1^* -  (Z_1'Z_1)^{-1} Z_1'Z_2  \hat{\theta}_2   -  a_2 \lrs{ \hat{\theta}_1^* - (Z_1'Z_1)^{-1} Z_1'Z_2  \hat{\theta}_2 }  - a_1 \hat{\theta}_1^*  + a_2 \hat{\theta}_1^* \\
& = (1-a_2)\hat{\theta}_1 + (a_2-a_1) \hat{\theta}_1^*
\end{align*}
we see that the blocked James--Stein-type shrinkage estimator for $\theta$ can be rewritten as
\begin{align*}
\hat{\theta}^{BJS} & =  (1-a_2)  \hat{\theta}
+  (a_2 - a_1) 
\begin{pmatrix}
 \hat{\theta}_1^* \\ \mathbf{0}
\end{pmatrix},
\end{align*}
where $\mathbf{0}$ denotes a $\abs{m_2}$-vector of zeros. Note that the first term of this estimator has the same structure as the usual James--Stein estimator with shrinkage factor $a_2$ (see the last display on page 33 in  \cite{huber13}). Using this, we can rewrite the true mean-squared prediction error as
\begin{align} \begin{split} \label{eq:pe.3}
\rho^2(m) & = ( (1-a_2) \hat{\theta} - \theta)' S ( (1-a_2) \hat{\theta} - \theta ) + (a_2 - a_1)^2 \hat{\theta}_1^*{'} S \hat{\theta}_1^* + s^2 \\
& \phantom{=} + 2 (a_2 - a_1) \hat{\theta}_1^*{'} [S_{1,1} ( (1-a_2) \hat{\theta}_1 - \theta_1) + S_{1,2} ( (1-a_2) \hat{\theta}_2 - \theta_2)].
\end{split}
\end{align}
We can rewrite the sum of the terms in the first line in \eqref{eq:pe.3} as
\begin{align}
\begin{split} \label{eq:pe.31}
& (1-a_2)^2(\hat{\theta}-\theta)'S(\hat{\theta} - \theta) + 2a_2(a_2-1)(\hat{\theta}-\theta)'S\theta + a_2^2 \theta'S\theta + s^2\\
& + (a_2-a_1)^2 (\hat{\theta}_1^{*} - \theta_1^*)' S_{1,1} (\hat{\theta}_1^* - \theta_1^*)  + 2 (a_2-a_1)^2 (\hat{\theta}_1^{*} - \theta_1^*)' S_{1,1} \theta_1^* + (a_2-a_1)^2 \theta_1^*{'} S_{1,1} \theta_1^*.
\end{split}
\end{align}
We can rewrite the sum of the terms in the second line in \eqref{eq:pe.3} as
\begin{align*}
2&(1-a_2)(a_2-a_1)(\hat{\theta}_1^*-\theta_1^*)'S_{1,1}(\hat{\theta}_1-\theta_1)  \\
& + 2(1-a_2)(a_2-a_1)(\hat{\theta}_1^*-\theta_1^*)'S_{1,2}(\hat{\theta}_2-\theta_2)  \\
&+2(1-a_2)(a_2-a_1)\theta_1^*{'}S_{1,1}(\hat{\theta}_1-\theta_1) +2(1-a_2)(a_2-a_1)\theta_1^*{'}S_{1,2}(\hat{\theta}_2-\theta_2)\\
&-2a_2(a_2-a_1)(\hat{\theta}_1^*-\theta_1^*)'(S_{1,1} \theta_1 +S_{1,2} \theta_2 )\\
&  -2a_2(a_2-a_1) \theta_1^*{'}(S_{1,1} \theta_1+S_{1,2} \theta_2).
\end{align*}
Using that $\hat{\theta}_1 - \theta_1 = \hat{\theta}_1^* - \theta_1^* - (Z_1'Z_1)^{-1} Z_1'Z_2 (\hat{\theta}_2 - \theta_2)$, we can rewrite the sum in the preceding display as
\begin{align}
\begin{split} \label{eq:pe.32}
2&(1-a_2)(a_2-a_1)(\hat{\theta}_1^*-\theta_1^*)'S_{1,1}(\hat{\theta}_1^*-\theta_1^*) \\
& +2(1-a_2)(a_2-a_1)(\hat{\theta}_1^*-\theta_1^*)'[S_{1,2} - S_{1,1}(Z_1'Z_1)^{-1}Z_1'Z_2](\hat{\theta}_2-\theta_2) \\
& +2(1-a_2)(a_2-a_1) (\hat{\theta}_1^*-\theta_1^*)'S_{1,1} \theta_1^* \\
& +  2(1-a_2)(a_2-a_1) \theta_1^*{'}[S_{1,2} - S_{1,1}(Z_1'Z_1)^{-1}Z_1'Z_2 ](\hat{\theta}_2-\theta_2) \\
&-2a_2(a_2-a_1)(\hat{\theta}_1^*-\theta_1^*)'(S_{1,1} \theta_1 +S_{1,2} \theta_2 )\\
&  -2a_2(a_2-a_1) \theta_1^*{'}(S_{1,1} \theta_1+S_{1,2} \theta_2).
\end{split}
\end{align}
Let $\tilde{Z}_2= Z_2 - Z_1 S_{1,1}^{-1} S_{1,2}$ and note that the quantity in squared brackets in \eqref{eq:pe.32} equals $-S_{1,1}(Z_1'Z_1)^{-1}Z_1'\tilde{Z}_2$. Using this and collecting the terms in \eqref{eq:pe.31} and \eqref{eq:pe.32}, we can rewrite the true mean-squared prediction error as
\begin{align}
\begin{split} \label{eq:pe.4}
\rho^2(m) & = (1-a_2)^2(\hat{\theta}-\theta)'S(\hat{\theta} - \theta)  + s^2\\
& \phantom{=} + (a_2 - a_1)(2 - a_1-a_2)(\hat{\theta}_1^* - \theta_1^*)'S_{1,1} (\hat{\theta}_1^* - \theta_1^*)   \\
& \phantom{=} +  a_2^2  \theta'S\theta + 2a_2(a_1-a_2) \theta_1^*{'}(S_{1,1} \theta_1+S_{1,2} \theta_2) +(a_2-a_1)^2 \theta_1^*{'} S_{1,1} \theta_1^* \\
& \phantom{=} + 2a_2(a_2-1)(\hat{\theta}-\theta)'S\theta \\
& \phantom{=} + 2 (1-a_1)(a_2-a_1) (\hat{\theta}_1^{*} - \theta_1^*)' S_{1,1} \theta_1^* \\ 
& \phantom{=} + 2 a_2(a_1 - a_2) (\hat{\theta}_1^* - \theta_1^*)'(S_{1,1} \theta_1 + S_{1,2} \theta_2 )\\
& \phantom{=} + 2(1-a_2)(a_1-a_2) (\hat{\theta}_2-\theta_2)' \tilde{Z}_2'Z_1 (Z_1'Z_1)^{-1} S_{1,1}\theta_1^*\\
& \phantom{=} + 2(1-a_2)(a_1 - a_2) (\hat{\theta}_1^* - \theta_1^*)'S_{1,1} (Z_1'Z_1)^{-1} Z_1' \tilde{Z}_2 (\hat{\theta}_2 - \theta_2).
\end{split}
\end{align}
Using the fact that $\theta_1^*= \theta_1+S_{1,1}^{-1} S_{1,2} \theta_2+(Z_1'Z_1)^{-1}Z_1'\tilde{Z}_2 \theta_2$, we have
\begin{align*}
\theta_1^*{'}(S_{1,1} \theta_1+S_{1,2} \theta_2) & = \theta_1'S_{1,1} \theta_1 + 2 \theta_1'S_{1,2} \theta_2+ \theta_2'S_{2,1} S_{1,1}^{-1} S_{1,2} \theta_2 \\
& \phantom{=} + (S_{1,1} \theta_1+S_{1,2} \theta_2)(Z_1'Z_1)^{-1} Z_1'\tilde{Z}_2 \theta_2
\end{align*}
as well as
\begin{align*}
\theta_1^*{'} S_{1,1} \theta_1^* &=  \theta_1'S_{1,1} \theta_1 + 2 \theta_1'S_{1,2} \theta_2+ \theta_2'S_{2,1} S_{1,1}^{-1} S_{1,2} \theta_2 \\
& \phantom{=} +2(S_{1,1} \theta_1+S_{1,2} \theta_2)(Z_1'Z_1)^{-1} Z_1'\tilde{Z}_2 \theta_2  \\
& \phantom{=} + \theta_2' \tilde{Z}_2'Z_1(Z_1'Z_1)^{-1}  S_{1,1} (Z_1'Z_1)^{-1} Z_1' \tilde{Z}_2 \theta_2.
\end{align*}
Noting that $ \theta_1'S_{1,1} \theta_1 + 2 \theta_1'S_{1,2} \theta_2+ \theta_2'S_{2,1} S_{1,1}^{-1} S_{1,2} \theta_2= \theta'S\theta - \theta_2'(S_{2,2} - S_{2,1} S_{1,1}^{-1} S_{1,2}) \theta_2$, the third line in \eqref{eq:pe.4} equals
\begin{align*}
 a_1^2  & (\theta'S\theta - \theta_2'(S_{2,2} - S_{2,1} S_{1,1}^{-1} S_{1,2}) \theta_2) + a_2^2 \theta_2'(S_{2,2} - S_{2,1} S_{1,1}^{-1} S_{1,2}) \theta_2 \\
 & + 2 a_1(a_1-a_2) (S_{1,1} \theta_1+S_{1,2} \theta_2)(Z_1'Z_1)^{-1} Z_1'\tilde{Z}_2 \theta_2  \\
&  + (a_2-a_1)^2  \theta_2' \tilde{Z}_2'Z_1(Z_1'Z_1)^{-1}  S_{1,1} (Z_1'Z_1)^{-1} Z_1' \tilde{Z}_2 \theta_2.
 \end{align*}
Using $\hat{\theta}_1 - \theta_1 = \hat{\theta}_1^* - \theta_1^* -S_{1,1}^{-1} S_{1,2} (\hat{\theta}_2 - \theta_2) - (Z_1'Z_1)^{-1} Z_1' \tilde{Z}_2 (\hat{\theta}_2 - \theta_2)$, we see that
\begin{align*}
(\hat{\theta} - \theta)'S \theta & = (\hat{\theta}_1 - \theta_1)'(S_{1,1} \theta_1 + S_{1,2} \theta_2) + (\hat{\theta}_2 - \theta_2)'(S_{2,1} \theta_1 + S_{2,2} \theta_2) \\
& = (\hat{\theta}_1^* - \theta_1^*)'(S_{1,1} \theta_1 + S_{1,2} \theta_2 ) -  (\hat{\theta}_2 - \theta_2)' \tilde{Z}_2'Z_1(Z_1'Z_1)^{-1} (S_{1,1} \theta_1 + S_{1,2} \theta_2 ) \\
& + (\hat{\theta}_2 - \theta_2)'(S_{2,2} -S_{2,1} S_{1,1}^{-1} S_{1,2})\theta_2.
 \end{align*}
Using the formula for $\theta_1^*$ as before in the fifth and seventh line in \eqref{eq:pe.4}, we can rewrite the true mean-squared prediction error as
\begin{align}
\begin{split} \label{eq:pe.5}
\rho^2(m) & = (1-a_2)^2(\hat{\theta}-\theta)'S(\hat{\theta} - \theta)  + s^2\\
& \phantom{=} + (a_2 - a_1)(2 - a_1-a_2)(\hat{\theta}_1^* - \theta_1^*)'S_{1,1} (\hat{\theta}_1^* - \theta_1^*)   \\
&  \phantom{=} + (a_2-a_1)^2  \theta_2' \tilde{Z}_2'Z_1(Z_1'Z_1)^{-1}  S_{1,1} (Z_1'Z_1)^{-1} Z_1' \tilde{Z}_2 \theta_2\\
& \phantom{=} +  a_1^2  (\theta'S\theta - \theta_2'(S_{2,2} - S_{2,1} S_{1,1}^{-1} S_{1,2}) \theta_2) + a_2^2 \theta_2'(S_{2,2} - S_{2,1} S_{1,1}^{-1} S_{1,2}) \theta_2 \\
& \phantom{=}  + 2 a_1(a_1-a_2) (S_{1,1} \theta_1+S_{1,2} \theta_2)'(Z_1'Z_1)^{-1} Z_1'\tilde{Z}_2 \theta_2  \\
& \phantom{=} + 2a_1(a_1 - 1) (\hat{\theta}_1^* - \theta_1^*)'(S_{1,1} \theta_1 + S_{1,2} \theta_2 ) \\
& \phantom{=} + 2(1-a_1)(a_2-a_1) (\hat{\theta}_1^* - \theta_1^*)'S_{1,1} (Z_1'Z_1)^{-1} Z_1' \tilde{Z}_2 \theta_2 \\
& \phantom{=} +2a_2(a_2-1) (\hat{\theta}_2 - \theta_2)'(S_{2,2} -S_{2,1} S_{1,1}^{-1} S_{1,2})\theta_2\\
& \phantom{=}+ 2 a_1(1-a_2)  (\hat{\theta}_2 - \theta_2)' \tilde{Z}_2'Z_1(Z_1'Z_1)^{-1} (S_{1,1} \theta_1 + S_{1,2} \theta_2 ) \\
& \phantom{=} + 2(1-a_2)(a_1-a_2) (\hat{\theta}_2-\theta_2)' \tilde{Z}_2'Z_1 (Z_1'Z_1)^{-1} S_{1,1}(Z_1'Z_1)^{-1} Z_1'\tilde{Z}_2 \theta_2 \\
& \phantom{=} + 2(1-a_2)(a_1 - a_2) (\hat{\theta}_1^* - \theta_1^*)'S_{1,1} (Z_1'Z_1)^{-1} Z_1' \tilde{Z}_2 (\hat{\theta}_2 - \theta_2).
\end{split}
\end{align}
In the preceding display, note that the terms in line five to ten follow, conditional on $X_1$ or $X$, respectively, a centered normal distribution with a variance that is bounded in probability, and it is easy to show that these terms converge to zero in probability. The term in the last line is of the form $w'Qw$ with $\trace(Q)=0$ and $w \sim N(0, s^2 I_{\abs{m_2}})$, and it is not hard to show that this term also converges to zero (see Lemma \ref{lem:qf.traceless} and the subsequent results).
The two results in this section gives some distributional properties of the terms involved in $\rho^2(m)$ and $\hat{\rho}^2(m)$ and motivate that those two quantities lie `close' to each other. 

For integers $k \geq 1$ and $d \geq 1$, let $\chi^2_k(\mu)$ denote a random variable that is chi-square distributed with $k$ degrees of freedom and noncentrality parameter $\mu \geq 0$, and let $W_{k}(S, d)$ denote a random $k \times k$ matrix that follows a Wishart distribution with scale matrix $S$ and $d$ degrees of freedom. We will write $\chi^2_k$ as shorthand for $\chi^2_k(0)$. Further, $\lambda_i(\cdot)$ denotes the $i$-th eigenvalue of the indicated matrix. Because we only consider eigenvalues of symmetric matrices, all eigenvalues are real and we assume that they are sorted in increasing order, i.e., $\lambda_1(\cdot) \leq \ldots \leq \lambda_d(\cdot)$ if $d$ is the dimension of the matrix. 

\begin{appxlem} \label{lem:dist}
Let the assumptions of this section hold. Let $\tilde{\theta}_1 = \theta_1+S_{1,1}^{-1} S_{1,2} \theta_2$, $b =\theta'S \theta$ and $b_2= \theta_2'(S_{2,2} - S_{2,1} S_{1,1}^{-1} S_{1,2}) \theta_2$.

\begin{enumerate}[label=(\roman*), leftmargin=0.7 cm, itemindent=0 cm]

\item \label{it:hot} The term $(\hat{\theta} -\theta)' S (\hat{\theta}-\theta)$ has the same distribution as $s^2$ times the ratio of two independent chi-square distributed random variables with $\abs{m}$ and $n - \abs{m} + 1$ degrees of freedom, respectively, i.e.
\begin{gather*}
(\hat{\theta} -\theta)' S (\hat{\theta}-\theta) \sim s^2 \frac{\chi^2_{\abs{m}}}{\chi^2_{n - \abs{m}+1}}.
\end{gather*}
The term $(\hat{\theta}_1^{*} -\theta_1^*)' S_{1,1} (\hat{\theta}_1^*-\theta_1^*)$ has the same distribution as the term in the preceding display with $\abs{m_1}$ instead of $\abs{m}$.
 
\item \label{it:s2} The estimator $\hat{s}^2$ has the same distribution as a chi-square distributed random variable with $n - \abs{m}$ degrees of freedom multiplied by $s^2$ and divided by $n - \abs{m}$, i.e.,
\begin{gather*}
\hat{s}^2 \sim s^2 \frac{\chi^2_{n - \abs{m}}}{n - \abs{m}}.
\end{gather*}
\item \label{it:nd} Let $\tilde{Z}_2=Z_2 - Z_1 S_{1,1}^{-1} S_{1,2}$. We have conditional on $Z_1$, 
\begin{align}
 \tilde{Z_2} \theta_2 & \sim N(0, b_2 I_n). \label{eq:nd1}
\end{align}

\item \label{it:ncc}  Let $P_1=Z_1(Z_1'Z_1)^{-1} Z_1$ and $M_1 = I_n - P_1$. We have conditional on $Z$,
\begin{align}
Y'P_1Y & \sim s^2 \chi^2_{\abs{m_1}}(\theta'Z'P_1 Z\theta/s^2), \label{eq:yp1y}
\intertext{as well as}
Y'M_1Y & \sim s^2 \chi^2_{n- \abs{m_1}}(\theta'Z'M_1Z\theta/s^2). \label{eq:ym1y}
\intertext{If $b_2 >0$, we have conditional on $Z_1$}
\theta' Z'P_1Z \theta &\sim b_2 \chi^2_{\abs{m_1}} (\tilde{\theta}_1'Z_1'Z_1 \tilde{\theta}_1/b_2). \label{eq:tzp1}
\intertext{If $b_2=0$, we have}
\theta' Z'P_1Z \theta = \theta_1'Z_1'Z_1 \theta_1 &\sim \theta_1'S_{1,1} \theta_1 \chi^2_{n}. \label{eq:tzp1.0}
\intertext{Furthermore, we have}
\theta'Z'M_1Z\theta &\sim b_2 \chi^2_{n-\abs{m_1}} \label{eq:tzm1}
\intertext{and}
\tilde{\theta}_1'Z_1'Z_1 \tilde{\theta}_1 & \sim (b-b_2) \chi^2_n. \label{eq:tt1z1}
\end{align}
If $\theta_1'S_{1,1} \theta_1 =0$, $b_2=0$ or $b-b_2=0$, respectively, the distributions should be understood as point mass at $0$. 

\item \label{it:qf} Conditional on $Z_1$, we have 
\begin{gather*}
Z_2'M_1Z_2 \sim W_{\abs{m_2}}(S_{2,2} - S_{2,1} S_{1,1}^{-1} S_{1,2}, n-\abs{m_1}).
\end{gather*}

\end{enumerate}
\end{appxlem}

\begin{proof}

By assumption, the rows of $Z$ are independent of each other and follow a normal distribution with mean vector zero and variance-covariance matrix $S$ that is a symmetric, positive definite and unknown $\abs{m} \times \abs{m}$ matrix. By assumption $Z_1$ and $Z_2$ are matrices whose independently distributed rows follow a centered normal distribution with variance-covariance matrix $S_{1,1}$ and $S_{2,2}$, respectively. Recall that $S_{2,2} - S_{2,1} S_{1,1}^{-1} S_{1,2}$ is the Schur complement of $S_{1,1}$ in $S$ and note that it is positive definite because $S$ is positive definite.

Proofs for the statements in \textit{\ref{it:hot}} and \textit{\ref{it:s2}} are well known but can be found in Lemma C.2 in \cite{huber13}. 

\begin{enumerate}[label= \it{(\roman*)}, leftmargin=0.7 cm, itemindent=0 cm]
\setcounter{enumi}{2}

\item Note that $\tilde{Z}_2$ has independent rows that follow a normal distribution with variance-covariance matrix $S_{2,2} - S_{2,1} S_{1,1}^{-1} S_{1,2}$ and that it is independent of $Z_1$. This fact is well known if $Z_1$ and $Z_2$ would be multivariate normal vectors. For the more general case of normal matrices, see for example Corollary 3.3.3.1 in \cite{mardia79}.

\item First of all, note that $P_1$ and $M_1$ are idempotent matrices with $\trace(P_1)=\abs{m_1}$ and $\trace(M_1)=n-\abs{m_1}$. Conditional on $Z$, we have $Y \sim N(Z \theta, s^2 I_n)$ which shows the first and the second statement. For the  statement in \eqref{eq:tzp1}, note that $Z \theta = Z_1 \tilde{\theta}_1+\tilde{Z}_2 \theta_2$ and use \eqref{eq:nd1} to conclude that, conditional on $Z_1$, $Z \theta \sim N(Z_1 \tilde{\theta}_1, b_2 I_n)$ for $b_2 > 0$. Note that $b_2=0$ implies that $\theta_2=0$ and hence that $ Z \theta = Z_1 \theta_1$. Noting that $Z_1 \theta_1 \sim N(0, \theta_1'S_{1,1} \theta_1 I_n)$ completes the statement in \eqref{eq:tzp1.0}.  For the statement in \eqref{eq:tzm1}, note that $\theta'Z'M_1 Z \theta = \theta_2'\tilde{Z}_2'M_1 \tilde{Z}_2 \theta_2$ and use the statement in \eqref{eq:nd1}. For the last statement, note that $Z_1 \tilde{\theta}_1 \sim N(0, \tilde{\theta}_1'S_{1,1} \tilde{\theta}_1 I_n)$ and that $\tilde{\theta}_1 S_{1,1} \tilde{\theta}_1 = b-b_2$.

\item Recall that $Z_2' M_1 Z_2 = \tilde{Z}_2'M_1 \tilde{Z}_2$ and the distribution of $\tilde{Z}_2$ discussed in the proof of \ref{it:nd}. The statement follows from Theorem 3.4.4 in \cite{mardia79}.

\end{enumerate}
\end{proof}

\begin{appxlem} \label{lem:dis.trace}
Let $V$ be a $d \times k$, $d \geq k$, random matrix where the entries are i.i.d.\ standard normally distributed random variables. Then each diagonal element of $(V'V)^{-1}$ follows an inverse chi-square distribution with $d-k+1$ degrees of freedom.
\end{appxlem}

\begin{proof}
Let $((V'V)^{-1})_{(i,j)}$ be the element in the $i$-th row and $j$-th column of $(V'V)^{-1}$. For the first diagonal element, partition the matrix $V$ as $V=(v_1, V_{-1})$, where $v_1$ is the first column of $V$ and $V_{-1}$ are the remaining $k-1$ columns of $V$, i.e., $v_1$ is a $d$-vector and $V_{-1}$ is a $d \times (k-1)$ matrix. This entails that
\begin{gather*}
V'V = 
\begin{pmatrix}
v_1'v_1  & v_1'V_{-1}\\
V_{-1}'v_1 & V_{-1}'V_{-1}
\end{pmatrix}.
\end{gather*}
The partitioned inversion rule (see, e.g., Appendix A.2.10 in \cite{johnstondinardo97}) shows that the first diagonal element of the inverse of $V'V$ is of the form $(v_1' v_1 - v_1'V_{-1}(V_{-1}'V_{-1})^{-1}V_{-1}'v_1)^{-1}$. This can be rewritten as the inverse of $v_1'Q v_1$, where $Q=I_n-V_{-1}(V_{-1}'V_{-1})^{-1}V_{-1}'$. Note that $Q$ is symmetric, idempotent and has rank $d-k+1$. By assumption $v_1 \sim N(0,I_d)$ and hence $v_1'Qv_1 \Vert V_{-1} \sim \chi^2_{d-k+1}$ (see, e.g, Appendix B.8 in \cite{johnstondinardo97}). Since the conditional distribution is independent of $V_{-1}$, the unconditional distribution coincides with the conditional distribution. This shows the statement for the first diagonal element. For the remaining diagonal elements, let $I_k^{i,j}$ be the $k \times k$ identity matrix with the $i$-th and the $j$-th column interchanged. Premultiplying a $k \times d$-matrix by $I_k^{i,j}$ interchanges the $i$-th and $j$-th row while postmultiplying a $d \times k$-matrix with $I_k^{i,j}$ interchanges the $i$-th and $j$-th column. Hence, 
\begin{gather*}
((V'V)^{-1})_{(a,a)} = (I_k^{1,a} (V'V)^{-1} I_k^{1,a})_{(1,1)} = ((I_k^{1,a} V'V I_k^{1,a})^{-1})_{(1,1)}
\end{gather*}
and the result follows from the first part of the proof.
\end{proof}

\section{Proof of Theorem \ref{th:help.bound}}

In this section, we first give auxiliary results that are needed for proofing Theorem \ref{th:help.bound}. The proof of the theorem is given at the end of this section. 

\begin{appxlem} \label{lem:bound.quadraticform}
For a fixed integer $d \geq 1$, let $v \sim N(0, I_d)$, and let $A$ be a symmetric, positive semidefinite and nonrandom $d \times d$ matrix. Then, we have for any $\eps >0$ that
\begin{align} \label{eq:pos.bound.quadraticform}
\P (v'Av - \trace(A)  \geq \eps) & \leq 
\begin{cases}
\e^{-\frac{d}{2} G(d \lambda_d(A), \eps)} & \text{if } \lambda_d(A)>0\\
0 & \text{else}
\end{cases}
\intertext{as well as}
\P ( v'Av - \trace(A)  \leq - \eps) & \leq 
\begin{cases}
\e^{-\frac{d}{2} G(d \lambda_d(A), - \eps)} & \text{if } \eps < \trace(A) \\
0 & \text{else},
\end{cases}
\label{eq:neg.bound.quadraticform}
\end{align}
where the function $G: (0, \infty) \times (- \trace(A), \infty) \to \R$ is given by $G(x,y)=y/x - \log((x+y)/x)$.
\end{appxlem}

\begin{proof}
Throughout the proof, we denote the eigenvalues of $A$ shorthand as $\lambda_1 \leq \ldots \leq \lambda_d$. Let $U\Lambda U'$ be the eigenvalue decomposition of $A$, where $U$ is the matrix whose columns are the normed eigenvectors of $A$ and where $\Lambda= \text{diag}(\lambda_1, \ldots, \lambda_d)$. If $\lambda_d=0$, then $A$ is the zero matrix and the statements of the lemma are trivially fulfilled. Hence, from now on we assume that $\lambda_d>0$.  Because $U$ is orthonormal, $U'w$ has the same distribution as $w$, i.e., $U'w \sim N(0,I_d)$. The left-hand side in \eqref{eq:pos.bound.quadraticform} can be rewritten as
\begin{align}
\P \lrr{ w' \Lambda w- \trace(\Lambda) \geq \eps }  & = \P \lrr{ \sum_{i=1}^d \lambda_i w_i^2- \sum_{i=1}^d \lambda_i \geq \eps } \label{eq:dis.qf}\\
& = \P \lrr{ \e^{ \sum_{i=1}^d s \lambda_i  w_i^2 } \geq \e^{ s \lrr{ \eps+\sum_{i=1}^d \lambda_i } } }, \nonumber
\end{align}
where $s>0$ is arbitrary. Use Markov's inequality to bound the last term in the preceding display from above by
\begin{align*}
\e^{ - s \lrr{ \eps+\sum_{i=1}^d \lambda_i } } \E \lrs{ \e^{ \sum_{i=1}^d s \lambda_i  w_i^2 }}   =  \e^{ -s \lrr{ \eps+\sum_{i=1}^d \lambda_i } } \prod_{i=1}^d \E \lrs{ \e^{ s \lambda_i  w_i^2 } } .
\end{align*}
Note that for every $i$, $\E[\exp(s \lambda_i w_i^2)] $ is the moment generating function of the chi-square distribution with one degree of freedom. The moment generating function is finite for $s \lambda_i < 1/2$ and equals $(1-2 s \lambda_i)^{-1/2} = \exp(-1/2 \log(1-2 s \lambda_i ))$. Hence, let $s>0$ be such that $ s \lambda_i  < 1/2$  for all $1 \leq i \leq d$ or, equivalently, such that $0 < s < 1/(2 \lambda_d)$. Then, the last term in the preceding display equals
\begin{align}
\exp & \lrr{- \frac{1}{2}  \sum_{i=1}^d \bigg[ 2s(\eps/d +  \lambda_i ) + \log(1 - 2 s \lambda_i ) \bigg]  } \label{eq:bound.quadraticform1} \\
& \leq \exp \lrr{- \frac{d}{2} \bigg[  2s(\eps/d +  \lambda_d ) + \log(1-2 s \lambda_d ) \bigg] } ,\label{eq:bound.quadraticform2}
\end{align}
where the inequality follows upon noting that the function in squared brackets in \eqref{eq:bound.quadraticform1} is nonincreasing in $\lambda_i$ for all $i$. (Indeed, note that the first derivative with respect to $\lambda_i$ equals $2s - 2s/(1-2s \lambda_i )$. The last term is nonpositive because $0 < 1- 2 \lambda_i s \leq 1$ as noted in the paragraph preceding the last display.) For choosing the optimal $s$ in \eqref{eq:bound.quadraticform2}, we need to maximize the function $f_{\lambda_d, \eps}(s)= 2s(\eps/d +  \lambda_d ) + \log(1-2 s \lambda_d )$. The first derivative with respect to $s$ equals $2 (\eps/d+\lambda_d) - 2 \lambda_d/(1-2 s \lambda_d )$. Setting the derivative equal zero and rearranging gives $1-2 s^* \lambda_d  = \lambda_d/(\eps/d+\lambda_d)$ which is equivalent to
\begin{gather*}
s^\ast = \frac{1}{2 \lambda_d} - \frac{1}{2(\eps/d + \lambda_d)} = \frac{\eps/d}{2 \lambda_d(\eps/d+\lambda_d)}.
\end{gather*}
Noting that $0 < s^\ast < 1/(2 \lambda_d)$ and that $\partial^2 f_{\lambda_d, \eps}(s)/\partial s^2=- 4 \lambda_d^2/(1-2 s \lambda_d )^2$, we see that indeed $s^\ast$ optimizes the upper bound in \eqref{eq:bound.quadraticform2}.
This shows the statement in \eqref{eq:pos.bound.quadraticform}. For the lower tail bound in \eqref{eq:neg.bound.quadraticform}, use the same arguments as before to show that
\begin{gather*}
\P \lrr{ v'Av - \trace(A) \leq - \eps}  =  \P \lrr{ \sum_{i=1}^d \lambda_i v_i^2 \leq   \sum_{i=1}^d \lambda_i - \eps }.
\end{gather*}
The term on the left-hand side in the preceding display equals zero if $\sum_{i=1}^d \lambda_i  - \eps =\trace(A) - \eps \leq 0$. Hence, the second statement in \eqref{eq:neg.bound.quadraticform} is fulfilled. Thus, consider the case where $\sum_{i=1}^d \lambda_i  - \eps = \trace(A) - \eps > 0$. Let $s>0$ be arbitrary and use Markov's inequality to show that the last term in the preceding display can be rewritten as and bounded from above by
\begin{align*}
\P \lrr{ \e^{ - s \sum_{i=1}^d \lambda_i v_i^2 } \geq \e^{ - s \lrr{ - \eps+\sum_{i=1}^d \lambda_i } } }&  \leq \e^{  s \lrr{ - \eps+\sum_{i=1}^d \lambda_i } }  \E \lrs{ \e^{ - \sum_{i=1}^d s\lambda_i  v_i^2 }}  \\
& =  \e^{ s \lrr{ \sum_{i=1}^d (\lambda_i-\eps/d) } } \prod_{i=1}^d \E \lrs{ \e^{ - s \lambda_i v_i^2 } } .
\end{align*}
As before $\E[ \exp( - s \lambda_i \chi^2_1)]$ is the moment generating function of the chi-square distribution with one degree of freedom which is finite for $- s \lambda_i < 1/2$ (which is fulfilled because $s >0$ and $\lambda_i \geq 0$ for all $i=1, \ldots, n$) and equals $(1+2 s \lambda_i )^{-1/2} = \exp(-1/2 \log(1+ 2 s \lambda_i ))$. We can rewrite the last term in the preceding display as
\begin{align}
\exp & \lrr{- \frac{1}{2}  \sum_{i=1}^d \bigg[ 2s(\eps/d - \lambda_i ) + \log(1 + 2 s \lambda_i ) \bigg] } \label{eq:bound.quadraticform3}
 \\ &
 \leq \exp \lrr{- \frac{d}{2} \bigg[  2s(\eps/d -  \lambda_d ) + \log(1+2 s \lambda_d ) \bigg] },  \label{eq:bound.quadraticform4}
\end{align}
where the inequality follows upon noting that the function in squared brackets in \eqref{eq:bound.quadraticform3} is nonincreasing in $\lambda_i$. (Note that the first derivative with respect to $\lambda_i$ equals $-2s+2s/(1+ 2 s \lambda_i )$. This is nonpositive as $1 \leq 1+2 s \lambda_i $.) In order to choose the optimal $s>0$, we need to maximize the function $f_{-\lambda_d,\eps}(s)$, where $f$ was defined before. Setting the first derivative with respect to $s$ equal zero and solving for $s$ gives by the same arguments used in the derivation of the upper tail bound
\begin{gather*}
s^{\ast \ast} = \frac{\eps/d}{2 \lambda_d(\lambda_d-\eps/d)}.
\end{gather*}
Noting that $\lambda_d-\eps/d$ is positive because $d \lambda_d \geq \sum_{i=1}^d \lambda_i > \eps$ here, we see that $ s^{\ast \ast} > 0$. The second derivative of $f_{-\lambda_d, \eps}(s)$ with respect to $s$ equals $- 4 \lambda_d^2/(1+2 \lambda_d s)^2$ which shows that $s^{\ast \ast}>0$ is a maximizer and ends the proof.
\end{proof}

\begin{appxcor} \label{cor:bound.quadraticform}
Under the assumptions of Lemma \ref{lem:bound.quadraticform}, we have for each $\eps>0$
\begin{align} \label{eq:bound.quadraticform}
\P \lrr{ \lrv{ v'Av - \trace(A)}  \geq \eps} & \leq 
\begin{cases}
2 \exp \lrr{ - d \frac{\eps^2}{4 d \lambda_d(A) (\eps + d \lambda_d(A) )}} & \text{if } \lambda_d(A)>0\\
0 & \text{else.}
\end{cases}
\end{align}
\end{appxcor}
	
\begin{proof}
The statement follows from Lemma \ref{lem:bound.quadraticform} by noting that $G(d \lambda_d(A), \eps)$ equals $M(0, \eps/(d \lambda_d(A)))$, where $M(\cdot, \cdot)$ is defined in Lemma C.3 in \cite{huber13}, together with Lemma C.4 in \cite{huber13}. The later lemma shows that $M(x,-y) \geq M(x,y)$ for $x \geq 0$ and $0 \leq y < 1$ and that $M(0,y) \geq y^2/(2(y+1))$ for $y \geq 0$. Hence, we have for $\eps < \trace(A)$ that $G(d \lambda_d(A), - \eps) =  M(0, - \eps/(d \lambda_d(A))) \geq M(0, \eps/(d \lambda_d(A))) = G(d \lambda_d(A) , \eps)$ because $\eps/(d \lambda_d(A)) \leq \eps/ \trace(A)< 1$ here, and that $M(0, \eps/(d \lambda_d(A))) \geq \eps^2/(2d \lambda_d(A)(\eps+d \lambda_d(A)))$.
\end{proof}

\begin{appxlem} \label{lem:qf.traceless}
Let $v \sim N(0, I_d)$, and let $A$ be a symmetric $d \times d$ matrix with $\trace(A)=0$. Then for each $\eps > 0$, we have
\begin{align*}
\P  & \lrr{ \lrv{ v'Av} \geq \eps } \\
& \leq  \begin{cases}
2 \exp\lrr{- \frac{d}{2} G(d \lambda_d(A), \eps/2) }  +  2 \exp\lrr{- \frac{d}{2} G(- d \lambda_1(A), \eps/2) }  & \text{ if } \lambda_d(A) >0 \\
 0 & \text{ else,}
\end{cases}
\end{align*}
where $G(\cdot, \cdot)$ is defined in Lemma \ref{lem:bound.quadraticform}.
\end{appxlem}

\begin{proof}
Because $\trace(A) = \sum_{i=1}^d \lambda_i(A) = 0$, we either have $\lambda_d(A)=0$ or $\lambda_d(A)>0$. If $\lambda_d(A)=0$, all eigenvalues are zero which implies that $A$ is the zero matrix (because it is symmetric) and the statement holds trivially. From now on, we assume that $\lambda_d(A) > 0$ implying that $\lambda_1(A) <0$. Let $U \Lambda U'$ be the eigenvalue decomposition of $A$.  Let $\Lambda^+$ be the matrix $\Lambda$ where all negative entries are substituted by 0, and let $\Lambda^-$ be the matrix $\Lambda$ where all nonnegative entries are substituted by 0. Note that $\Lambda = \Lambda^+ + \Lambda^-$. Hence, we have
\begin{align*}
v'Av - \trace(A) = v'U \Lambda^+ U'v - \trace(\Lambda^+)+ v'U \Lambda^- U' v - \trace(\Lambda^-).
\end{align*}
Because $U$ is orthonormal, we see that $U'v$ has the same distribution as $v$, i.e., $U'v \sim N(0, I_d)$. The left-hand side of the statement can then be rewritten as
\begin{align*}
\P & \lrr{\lrv {v'Av - \trace(A)} \geq \eps } \\ 
& \leq \P \lrr{ \lrv{ v' \Lambda^+ v - \trace(\Lambda^+)}\geq \eps/2}  
 + \P \lrr{ \lrv{ v' \Lambda^- v - \trace(\Lambda^-)}\geq \eps/2} \\
& = \P \lrr{ \lrv{ v' \Lambda^+ v - \trace(\Lambda^+)}\geq \eps/2}  
+ \P \lrr{ \lrv{ v' (-\Lambda^-) v - \trace(-\Lambda^-)}\geq \eps/2}.
\end{align*}
Because $\Lambda^+$ as well as $-\Lambda^-$ are positive semidefinite by construction and because $\lambda_d(\Lambda^+) = \lambda_d(A) >0$ and $\lambda_d(-\Lambda^-) = - \lambda_1(A)>0$, the result follows from Lemma \ref{lem:bound.quadraticform}.
\end{proof}

\begin{appxlem} \label{lem:qf.traceless0}
Let $v \sim N(0, I_d)$, and let $A$ be a $d \times d$ matrix such that $\trace(A)=0$ and $AA= \mathbf{0}$, i.e., the $d \times d$ matrix consisting of only zeros. Then, we have for every $\eps > 0$
\begin{align*}
\P  \lrr{ \lrv{ v'Av} \geq \eps } \leq  \begin{cases}
 4 \exp\lrr{- \frac{d}{2} G(d \sqrt{\lambda_d(A'A)/2}, \eps/2) }  & \text{ if } \lambda_d(A+A') > 0 \\
 0 & \text{ else}
\end{cases}
\end{align*}
where $G(\cdot, \cdot)$ is defined in Lemma \ref{lem:bound.quadraticform}.
\end{appxlem}

\begin{proof}
Let $A'$ be the transpose of $A$ and set $B = (A+A')/2$. Then, $B$ is symmetric, $\trace(B)=0$ and $v'Av = w'Bw$. Note that $\lambda_i(B) =  \lambda_i(A+A')/2$ for all $i=1, \ldots, d$. If $\lambda_d(B)=0$, then all eigenvalues of $B$ are 0 (because $\trace(B)=0$) implying that $B$ is the zero matrix (because $B$ is symmetric by construction). Noting that $\lambda_d(B)=0$ if and only if $\lambda_d(A+A')=0$, the second line of the statement is trivially fulfilled. If $\lambda_d(B) = \lambda_d(A+A')/2 >0$, we can use Lemma \ref{lem:qf.traceless} to bound the quantity of interest from above by
\begin{gather*}
2 \exp\lrr{- \frac{d}{2} G( d \lambda_d(B), \eps/2) } +  2 \exp\lrr{- \frac{d}{2} G( - d \lambda_1(B), \eps/2)}.
\end{gather*}
Note that $(\lambda_i(A+A'))^2=\lambda_j((A+A')^2)$ for some $i$ and $j$, that $(A+A')^2= A'A + AA'$ by assumption and that $A'A$ and $AA'$ have the same eigenvalues. Hence $0 < \lambda_d(B) \leq 1/2 \sqrt{\lambda_d(A'A+AA')} \leq  1/2 \sqrt{2 \lambda_d(A'A)}$ by Weyl's inequality. Using the same arguments as for $\lambda_d(B)$, we have $0 < - \lambda_1(B) \leq 1/2 \sqrt{2 \lambda_d(A'A)}$. The result follows from the fact that the function $G(\cdot , \cdot)$ is nonincreasing in its first argument.
\end{proof}

\begin{appxcor} \label{cor:qf.traceless0}
Under the assumptions of Lemma \ref{lem:qf.traceless0}, we have for each $\eps >0$
\begin{align*}
\P  \lrr{ \lrv{ v'Av} \geq \eps } \leq  \begin{cases}
 4 \exp\lrr{- d \frac{\eps^2}{4 d \sqrt{2 \lambda_d(A'A)} (\eps + d \sqrt{ 2 \lambda_d(A'A)} )  }  }  & \text{if } \lambda_d(A+A') > 0 \\
 0 & \text{else.}
\end{cases}
\end{align*}
\end{appxcor}

\begin{proof}
The statement follows from Lemma \ref{lem:qf.traceless0} by the same arguments as Corollary \ref{cor:bound.quadraticform} follows from Lemma \ref{lem:bound.quadraticform}.
\end{proof}

The next two results are standard. We state it only for the sake of completeness. 
\begin{appxlem} \label{lem:nd}
Let $w \sim N(0, \tau^2)$ with $\tau^2 >0$. Then, we have for every $\eps >0$
\begin{gather*}
\P \lrr{ \lrv{w} \geq \eps} \leq 2 \exp \lrr{- \eps^2/(2 \tau^2)}.
\end{gather*}
\end{appxlem}

\begin{appxlem} \label{lem:chi2}
We have for every $k \in \N$ and every $\eps >0$
\begin{align*}
\P \lrr{ \chi^2_k/k - 1  \geq \eps} & \leq  \exp \lrr{ - k \eps^2/(4(\eps+1)) } \\
\intertext{as well as}
\P \lrr{ \lrv{ \chi^2_k/k - 1 } \geq \eps} & \leq 2 \exp \lrr{ - k  \eps^2/(4(\eps+1)) }.
\end{align*}
%
\end{appxlem}

\begin{appxlem} \label{lem:chi2.balance}
Fix integers $d \geq k \geq 1$. Let $\chi^2_k(B_d)$ and $B_d > 0$ be real random
variables such that $\chi^2_k(B_d)$ follows conditional
on $B_d$ a noncentral chi-square distribution with $k$ degrees of freedom and 
noncentrality parameter $B_d$. Then for any $\eps > 0$, we have
\begin{align} 
\begin{split} \label{eq:chi2.balance.abs}
\P \lrr{ \lrv{ \frac{\chi_k^2(B_d)}{d}-\frac{k+B_d}{d} }\geq \eps} &  \leq   2 \E \lrs{ \exp \lrr{ - d \frac{\eps^2  }{4( \eps + 2(k/d + B_d/d))} }}
\end{split}
\intertext{as well as} 
\begin{split} \label{eq:chi2.balance}
\P \lrr{ \frac{\chi_k^2(B_d)}{d}-\frac{k+B_d}{d} \geq \eps} &  \leq   \E \lrs{ \exp \lrr{ - d \frac{\eps^2  }{4( \eps + 2(k/d + B_d/d ))} }}.
\end{split}
\end{align}
\end{appxlem}

\begin{proof}
Rewrite the left-hand side in \eqref{eq:chi2.balance.abs} as
\begin{gather*} 
 \E \lrs{ \P \lrr{ \lrv{ \frac{\chi_{k}^2(B_d)}{k+B_d}-1}\geq \eps \frac{d}{k+B_d}  \Big \Vert B_d } } 
\end{gather*}
and use the Chernoff bound for the noncentral chi-square distribution as in Corollary C.5 in \cite{huber13}, conditional on $B_d$, to bound the term in the preceding display from above by
\begin{align}
2 \E \lrs{ \exp \lrr{ - k  \frac{  \eps^2  d^2/(k(k+B_d))}{4 \left(  \eps d/(k+B_d)+2 \right)} }}  = 2 \E \lrs{ \exp \lrr{ - d \frac{\eps^2  d}{4( \eps  d+2(k+B_d))} } }.
\end{align}
The proof of the second statement is the same but using the one-sided version of the tail bound of the noncentral chi-square distribution.
\end{proof}

\begin{appxlem} \label{lem:eigenvalues}
Let $W$ follow a Wishart distribution with scale matrix $I_k$ and $d \geq k \geq 2$ degrees of freedom, i.e., $W \sim W_k(I_k, d)$. Then, we have for every $\gamma_1 \in [0,1)$
\begin{gather*}
\P \lrr{ \lambda_1(W/d) \leq \gamma_1^2(1-\sqrt{k/d})^2} \leq \exp \lrr{ - d (1 - \gamma_1)^2(1-\sqrt{k/d})^2/2}.
\end{gather*}
Furthermore, we have for every $\eps > 0$
\begin{gather*}
\P \lrr{ \lambda_d(W/d) \geq (1+\sqrt{k/d}+\eps)^2} \leq \exp \lrr{ - d \eps^2/2}.
\end{gather*}
Especially, we have for every $\gamma_2 > 0$
\begin{gather*}
\P \lrr{ \lambda_d(W/d) \geq (1+\gamma_2)^2(1+\sqrt{k/d})^2} \leq \exp \lrr{ - d  \gamma_2^2(1+\sqrt{k/d})/2}.
\end{gather*}
\end{appxlem}

\begin{proof}
The lemma follows immediately from Theorem 2.13 in \cite{davidsonszarek01}. A detailed proof of the first statement can be found in \cite{huber13} (see Lemma C.11, Corollary C.12 and the following remark). 
\end{proof}

\begin{appxlem} \label{lem:bound.trace}
Let $V$ be a random $d \times k$ matrix, $d \geq k$, with i.i.d.\ standard
normally distributed entries. Then for $\eps >0$, we have
\begin{align} 
 \label{eq:bound.trace.rel}  
\P \lrr{ \lrv{ \frac{ \trace\lrr{(V'V)^{-1}} }{k/(d-k+1)} -1 } \geq \eps} 
& \leq 2 k \exp
  \lrr{ - (d-k) \frac{\eps^2}{8( \eps + 1)^2}},\\
  \intertext{as well as} \label{eq:bound.trace.abs}
\P \lrr{ \lrv{  \trace\lrr{(V'V)^{-1}} - \frac{k}{d-k+1} } \geq \eps} 
& \leq 2 k \exp
  \lrr{ - (d-k) \frac{\eps^2(1-k/d)^2}{8( \eps(1-k/d) + 1)^2}},
\end{align}
where $\trace((V'V)^{-1})$ is to be interpreted as 1 if $V'V$ is not invertible.
\end{appxlem}

\begin{proof}
It suffices to consider only the case where $V'V$ has full rank as this happens with probability 1. Lemma~\ref{lem:dis.trace} shows that the trace of $(V'V)^{-1}$ is the sum of $k$ random variables that are not independent and where each follows an inverse $\chi^{2}_{d-k+1}$-distribution. Hence, the left-hand side in \eqref{eq:bound.trace.rel} equals
\begin{gather*}
 \P  \lrr{ \lrv{ \sum_{i=1}^k \lrr{ \frac{d-k+1}{\chi^2_{d-k+1}}  - 1 }} \geq \eps k }
 \end{gather*}
Use the triangle inequality to bound the sum in the preceding display from above by
\begin{align*} 
k &\P \lrr{ \lrv{  \frac{d-k+1}{\chi^2_{d-k+1}}  - 1 } \geq \eps } 
 = k \P \lrr{   \frac{d-k+1}{\chi^2_{d-k+1}}    \geq 1+ \eps } + k \P \lrr{ \frac{d-k+1}{\chi^2_{d-k+1}}    \leq 1 - \eps}.
\end{align*}
Note that the term on the far right-hand side equals zero if $\eps \geq 1$. The sum on the right hand-side in the preceding display can be rewritten as
\begin{gather*} 
k \P \lrr{   \frac{\chi^2_{d-k+1}}{d-k+1}  - 1  \leq - \frac{\eps}{1+\eps} } + k \P \lrr{ \frac{\chi^2_{d-k+1}}{d-k+1}  - 1  \geq  \frac{\eps}{1-\eps} },
\end{gather*}
where the second term is to be interpreted as zero if $\eps \geq 1$. Noting that $\eps/(1- \eps) \geq \eps/(1+\eps)$ for $\eps \in (0, 1)$, we can bound the sum in the preceding display from above by 
\begin{gather*} 
k \P \lrr{  \lrv{ \frac{\chi^2_{d-k+1}}{d-k+1}  - 1 } > \frac{\eps}{1+\eps} }.
\end{gather*}
Use the tail bound for the chi-square distribution as in Lemma~\ref{lem:chi2} to bound the term in the preceding display from above by
\begin{align*}
 2  k \exp\ \lrr{ - (d-k+1) \frac{\eps^2}{4(1+\eps)(2\eps+1)} },
\end{align*}
which is further bounded from above by the right-hand side in \eqref{eq:bound.trace.rel}. The second statement follows from the first statement by rewriting it as
\begin{align*}
\P \lrr{ \lrv{ \frac{ \trace\lrr{(V'V)^{-1}} }{k/(d-k+1)} -1 } \geq \eps \frac{d-k+1}{k}}
\end{align*}
and noting that $\eps(d-k+1)/k \geq \eps (1-k/d)$.
\end{proof}

The next result is rather technical and is separated from the proof of the theorem for the better readability.

\begin{appxlem} \label{lem:prop}

\begin{enumerate}[label=(\roman*), leftmargin=0 cm, itemindent=0.5 cm]

\item Let $x_1, x_2 \in [0,1]$ and $y_1, y_2 \geq 0$ and let $k_1, k_2, k, d \in \N$ such that $k_1 + k_2 = k < d$. Let 
\begin{align*}
q & = (1- x_2)^2 \frac{k}{d-k+1}+(x_2-x_1)(2-x_1-x_2) \frac{k_1}{d - k_1+1} +1+x_1^2 y_1 + x_2^2 y_2 \\
& \phantom{=} +(x_1-x_2)^2 \frac{k_1}{d-k_1 +1} y_2.
\end{align*}
Then, we have the following inequalities:
\begin{align} 
\Bigg \vert (1- x_2)^2 \frac{k}{d-k+1}+(x_2-x_1)(2-x_1-x_2) & \frac{k_1}{d - k_1+1} \nonumber\\
 +1-x_2^2-(x_1-x_2)^2  \frac{k_1}{d-k_1+1} \Bigg \vert/q  & \leq \frac{1}{1-k/d} \label{eq:ineq1} \\
x_1^2/q & \leq \frac{1}{1+y_1+ y_2 k_1/d(1-k_1/d)} \leq 1 \label{eq:ineq2} \\
\lrv{ x_2^2-x_1^2 \frac{k_1}{d} + (x_1 - x_2)^2 \frac{k_1}{d - k_1+1} } / q  &\leq \frac{1}{(1-k_1/d)(1+y_2)} \leq \frac{1}{1-k_1/d} \label{eq:ineq3}
\end{align}

\item For $x \in [0,1]$, the following inequalities hold true:
\begin{align}
2(1-x)^5/9 & \leq (1-\sqrt{x})^4 \label{eq:prop1}\\ 
3(1-x)^5/4  & \leq (1-\sqrt{x})^2. \label{eq:prop2}\\ 
\intertext{For $x,y \in (0,1)$ such that $x + y \leq 1$, we have the following inequalities:}
3 (1-x-y)^5/4  & \leq (1-\sqrt{y/(1-x)})^2(1-x) ,\label{eq:propxy1} \\ 
(1-x-y)^5/2  & \leq (1-\sqrt{y/(1-x)})^2 (1-\sqrt{x})^2(1-x) , \label{eq:propxy2} \\ 
8 (1-x-y)^5/45  & \leq (1-\sqrt{y/(1-x)})^2(1-\sqrt{x})^4, \label{eq:propxy3} \\ 
8(1-x-y)^5/165  & \leq (1-\sqrt{y/(1-x)})^2(1-\sqrt{x})^4 (1-x), \label{eq:propxy4} \\ 
3 (1-x-y)^5/4  & \leq (1-\sqrt{y/(1-x)})^2. \label{eq:propxy5} 
\end{align}
\end{enumerate}

\end{appxlem}

\begin{proof}

\begin{enumerate}[label=(\roman*), leftmargin=0 cm, itemindent=0.5 cm]
\item 
We will show that the left hand-side in \eqref{eq:ineq1} is bounded from above by $d/(d-k+1)$ which implies the statement.
Bound $q$ from below by 
\begin{align*} 
q_1=(1-x_2)^2 \frac{k}{d - k + 1}+(x_2-x_1)(2-x_1-x_2) \frac{k_1}{d - k_1 + 1} +1.
\end{align*}
We need to show that the term in absolute value on the left hand-side in \eqref{eq:ineq1} multiplied by $(d-k+1)/d$ is bounded from above by $q_1$. Hence, we have to show that 
\begin{align*}
&\lrr{ (1-x_2)^2\frac{k}{d - k + 1}+(x_2-x_1)(2-x_1-x_2) \frac{k_1}{d - k_1 + 1}} \lrr{1 - \frac{d-k+1}{d} }\\
&+ 1 - \frac{d-k+1}{d}  +x_2^2 \frac{d-k+1}{d}  +(x_1-x_2)^2 \frac{k_1}{d} \frac{d-k+1}{d-k_1+1}\\
\intertext{as well as}
&(1-x_2)^2 \frac{k}{d - k + 1} +(x_2-x_1)(2-x_1-x_2) \frac{k_1}{d - k_1 + 1} \\
&+(1-x_2)^2 \frac{k}{d} +(x_2-x_1)(2-x_1-x_2) \frac{k_1}{d} \frac{d-k+1}{d-k_1+1} \\
& + 1 - (x_1-x_2)^2 \frac{k_1}{d} \frac{d-k+1}{d-k_1+1} +(1-x_2^2) \frac{d-k+1}{d}
\end{align*}
is nonnegative. The term in the second-to-last display is nonnegative because $k/(d-k+1) \geq k_1/(d-k_1+1)$, $(1-x_2)^2+(x_2-x_1)(2-x_1-x_2) = (1-x_1)^2\geq 0$ and $1-(d-k+1)/d \geq 0$. To show that the term in the preceding display is nonnegative, use $k/(d-k+1) \geq k_1/(d-k_1+1)(1-k_2/d)$ for the first term in the first line, that $k/d \geq k_1/d$ for the first term in the second line and that $(d-k+1)/(d-k_1+1) = 1 - k_2/(d-k_1+1)$ for the second term in the second line. Rearranging the terms, it is enough to show that 
\begin{align*}
& \frac{k_1}{d - k_1+1} \lrr{ 1- \frac{k_2}{d}} \lrr{(1-x_2)^2+(x_2-x_1)(2-x_1-x_2)} \\
& + \frac{k_1}{d} \lrr{ (1-x_2)^2+(x_2-x_1)(2-x_1-x_2) } + (1-x_2^2) \frac{d-k+1}{d}\\
& + 1-(x_1-x_2)^2 \frac{k_1}{d} \frac{d-k+1}{d-k_1+1} 
\end{align*}
is nonegative.
The inequality in the third line in the preceding display holds because $k_1/d \leq 1$, $(d-k+1)/(d-k_1+1) \leq 1$ and $x_1, x_2 \in [0,1]$.\\
To show the inequality in \eqref{eq:ineq2}, note that $q$ is bounded from below by $1+ x_1^2 y_1+ x_2^2 y_2+(x_1-x_2)^2y_2  k_1/d $ because $k/(d-k+1) \geq k_1/(d-k_1+1) \geq k_1/d$ and $(1-x_2)^2+(x_2-x_1)(2-x_1-x_2) = (1-x_1)^2\geq 0$. Hence, we need to show that 
\begin{align*}
1 -  x_1^2 + y_2 \lrs{ x_2^2 - x_1^2 \frac{k_1}{d} \lrr{1-\frac{k_1}{d}} + (x_1-x_2)^2 \frac{k_1}{d} } \geq 0.
\end{align*}
Rewriting the sum in squared brackets as $x_2^2 k_1/d + (x_2-x_1 k_1/d)^2$ shows the statement.\\
The second inequality in \eqref{eq:ineq3} is clear. To show the first inequality in \eqref{eq:ineq3}, we will show that the term on the far left-hand side is bounded from above by $1/(1-k_1/d+y_2)$ which implies the first inequality. The second inequality is immediate by recalling that $y_2 \geq 0$. Note that $q$ is bounded from below by
\begin{align*}
q_2 =  (1-x_1)^2 \frac{k_1}{d - k_1 + 1}  + 1 + x_2^2 y_2+(x_1-x_2)^2 \frac{k_1}{d-k_1+1} y_2 
\end{align*}
because $k/(d-k+1) \geq k_1/(d-k_1+1) \geq 0$. Thus, it suffices to show that 
\begin{gather*} 
-q_2 \leq \lrr{ x_2^2-x_1^2 \frac{k_1}{d}+(x_1-x_2)^2 \frac{k_1}{d-k_1+1} } \lrr{1-\frac{k_1}{d}+y_2} \leq q_2
\end{gather*}
or, equivalently, that 
\begin{align}
\begin{split} \label{eq:ineq3.lower}
(1-x_1)^2  \frac{k_1}{d-k_1+1}+ \lrr{ x_2^2 +(x_1-x_2)^2 \frac{k_1}{d-k_1+1} } \lrr{1-\frac{k_1}{d}}  \\
+1-x_1^2 \frac{k_1}{d} \lrr{1-\frac{k_1}{d}} + y_2 \lrs{ 2 x_2^2 + 2 (x_1-x_2)^2 \frac{k_1}{d-k_1+1}   -x_1^2 \frac{k_1}{d}} &\geq 0
\end{split}
\intertext{as well as}
\begin{split} \label{eq:ineq3.upper}
1- x_2^2\lrr{1-\frac{k_1}{d}} + x_1^2 \frac{k_1}{d} \lrr{1-\frac{k_1}{d}} - (x_1-x_2)^2 \frac{k_1}{d-k_1+1} \lrr{1-\frac{k_1}{d}} \\
+ (1-x_1)^2  \frac{k_1}{d-k_1+1}+x_1^2 \frac{k_1}{d} y_2 & \geq 0
\end{split}
\end{align}
holds. To show the inequality in \eqref{eq:ineq3.lower}, note that the sum of the first four terms is nonnegative because $x_1, k_1/d  \in [0,1]$. Hence, we are left with showing that the sum of the terms in squared brackets is nonnegative. Using that $k_1/(d-k_1+1) \geq k_1/d$ and that $1 \geq k_1/d$, it is enough to show that $2 x_2^2+2(x_1-x_2)^2 -x_1^2 \geq 0$. But the left hand-side of the preceding inequality equals $4 x_2^2-4 x_1 x_2+x_1^2 =(2 x_2-x_1)^2 \geq 0$. To show the inequality in \eqref{eq:ineq3.upper}, note that the sum in the second line is nonnegative. Using $(1-k_1/d) k_1/(d-k_1+1) \leq k_1/d$, we can bound the sum in the first line from below by $1-x_2^2-x_1^2 (k_1/d)^2 +2 x_1 x_2 k_1/d=1-(x_1 k_1/d -x_2)^2$. But this is nonegative because $x_1, x_2, k_1/d \in [0,1]$.

\item Define $f_{a,b}(s,t) = (s- \sqrt{t})^a(s + \sqrt{t})^b$ for $a, b \in \N$ with $a < b$ and $s, t \in \R$ with $s>0$ and $0 < t \leq s^2$. Note that the function is nondecreasing in $s$ and that
\begin{align*}
\frac{\partial f_{a,b}(s,t)}{\partial t} & = \frac{(s-\sqrt{t})^{a-1}(s+\sqrt{t})^{b-1}}{2 \sqrt{t}}(  s(b-a) - \sqrt{t} (a+b)) \\
\intertext{and that}
\frac{\partial^2 f_{a,b}(s,t)}{\partial t^2} & = \frac{(s-\sqrt{t})^{a-2}(s+\sqrt{t})^{b-2}}{4 t \sqrt{t}} \Big[ s^3 (a-b) + s^2 \sqrt{t} (a-b)^2 \\
&  \phantom{==} + s t (a-b)(2 a + 2 b-3)+t \sqrt{t}(a+b)(a+b-2) \Big].
\end{align*}
The functions $f_{a,b}(s,t)$ have an extremum at $t^*= s^2(b-a)^2/(a+b)^2$
with the value 
\begin{align} \label{eq:max.f}
f_{a,b}(s,t^*) = a^a b^b (2 s/(a+b))^{a+b}.
\end{align}
This a maximizer because the term inside the squared brackets in the second-to-last display evaluated at $t^*$ equals $ - 4 a b (b-a)   s^3/(a+b)^2$ which is negative because $a < b$ by assumption.

The inequalities in \eqref{eq:prop1} and \eqref{eq:prop2} are equivalent to $9/2 \geq (1-\sqrt{x})(1+ \sqrt{x})^5 = f_{1,5}(1,x)$ and to $4/3 \geq (1-\sqrt{x})^3(1+\sqrt{x})^5=f_{3,5}(1,x)$. Using the formula in \eqref{eq:max.f}, we see that $f_{1,5}(1,x) \leq 5^5/3^6$ and that $f_{3,5}(1,x) \leq 3^3 5^5/4^8$ which shows the statement.

In order to show the inequalities in \eqref{eq:propxy1}, \eqref{eq:propxy2}, \eqref{eq:propxy3} and \eqref{eq:propxy4}, it is enough to show that $(\sqrt{1-x}-\sqrt{y})^3(\sqrt{1-x} + \sqrt{y})^5=f_{3,5}(\sqrt{1-x},y)$ is bounded from above by $4/3$, $2 (1-\sqrt{x})^2$, $45(1-\sqrt{x})^4/(8(1-x))$ and by $165(1-\sqrt{x})^4/8$, respectively. Note that by \eqref{eq:max.f} we have $f_{3,5}(\sqrt{1-x}, y) \leq 3^3 5^5(1-x)^4/4^8$. For the inequality in \eqref{eq:propxy1}, we need to show that $3^3 5^5 (1-x)^4/4^8 \leq 4/3$ which holds because $x \in [0,1)$. For \eqref{eq:propxy2}, we are left with showing that $2^{17}/(3^3 5^5) \geq (1-\sqrt{x})^2(1+\sqrt{x})^4 =f_{2,4}(1, x)$. The right hand-side of the preceding inequality is bounded from above by $2^{10}/3^6$ by \eqref{eq:max.f} which shows the statement. For \eqref{eq:propxy3}, it suffices to show that $2^{13}/(3 \cdot 5^4) \geq (1-\sqrt{x})(1+\sqrt{x})^5 = f_{1,5}(1,x)$. By \eqref{eq:max.f}, $f_{1,5}(1,x)$ is bounded from above by $5^5/3^6$. For inequality in \eqref{eq:propxy4}, we have to show that $(1+\sqrt{x})^4 \leq 11 \cdot 2^{13}/(3^2 5^4)$ which is true because $(1+\sqrt{x})^4 \leq 2^4$ by assumption.

The inequality in \eqref{eq:propxy5} follows immediately from \eqref{eq:propxy1} because the right-hand side in \eqref{eq:propxy5} is bounded from below by the right hand-side in \eqref{eq:propxy1}.

\end{enumerate}

\end{proof}

The next two results collect the convergence rates of the main terms in the proof of Theorem~\ref{th:help.bound}. In both results, we look at one fixed model $m$.
\begin{appxlem} \label{lem:bound.rhohat}
For each $\delta>0$, we have
\begin{align} \label{eq:bound.rhohat.pos}
\P(\hat{\rho}^2(m)/r > \exp(\delta)) & \leq 22 \exp \lrr{- \abs{m_1} \lrr{1- \frac{\abs{m}}{n}}^3 \frac{(\ed-1)^2}{210 \ed} }  \\
\intertext{as well as} \label{eq:bound.rhohat.neg}
\P( \hat{\rho}^2(m)/r < \exp(-\delta)) & \leq 22 \exp \lrr{- \abs{m_1} \lrr{1- \frac{\abs{m}}{n}}^3 \frac{(\ed-1)^2}{840 \exp(2 \delta)} },
\end{align}
where 
\begin{align}
\begin{split} \label{eq:r}
r =  & s^2 \Bigg( (1-a_2)^2 \frac{\abs{m}}{n - \abs{m} +1} +(a_2-a_1)(2-a_1-a_2) \frac{\abs{m_1}}{n - \abs{m_1} +1} + 1 \\
& + a_1^2 ( \mu - \mu_2)  +  a_2^2  \mu_2 + (a_1 - a_2)^2 \frac{\abs{m_1}}{n - \abs{m_1}+1} \mu_2  \Bigg),
\end{split}
\end{align}
where $a_1$ and $a_2$ are the shrinkage factors as in \eqref{eq:a1a2} and where $\mu=\theta'S\theta/ s^2$ and $\mu_2 = \theta_2'(S_{2,2} - S_{2,1} S_{1,1}^{-1} S_{1,2})\theta_2/s^2$.
\end{appxlem}

\begin{proof}
Let $P_1 = Z_1(Z_1'Z_1)^{-1} Z_1$ and $M_1=I_n - P_1$ and rewrite $\hat{\rho}^2(m)/r$ as 
\begin{align*}
&  r_1  \lrr{ \frac{\hat{s}^2}{s^2} - 1} + r_2 \lrr{ \frac{Y'P_1 Y/s^2}{n} - \frac{\abs{m_1}}{n} - \mu_2 \frac{\abs{m_1}}{n} - (\mu-\mu_2) }  \\& 
  + r_3 \lrr{ \frac{Y'M_1Y/s^2 }{n-\abs{m_1}} - 1 -  \mu_2 }  + 1
\end{align*}
with
\begin{align*}
r_1 &= s^2 \Bigg( (1-a_2)^2 \frac{\abs{m}}{n-\abs{m}+1} + (a_2 - a_1)(2 - a_1 - a_2)  \frac{\abs{m_1}}{n-\abs{m_1}+1}  \\
& \phantom{= \hat{\sigma^2(m)}^2 \; \; } +1 - a_2^2 - (a_2- a_1)^2 \frac{\abs{m_1}}{n-\abs{m_1}+1} \Bigg)/r, \\
r_2 &= s^2 a_1^2/r,  \\
r_3 &= s^2 \lrr{ a_2^2 - a_1^2 \frac{m_1}{n} + (a_1- a_2)^2 \frac{\abs{m_1}}{n-\abs{m_1}+1}}/r.
\end{align*}
Use Lemma~\ref{lem:prop}~(i) with $x_1=a_1$, $x_2=a_2$, $y_1=\mu-\mu_2$, $y_2=\mu_2$, $k_1=\abs{m_1}$, $k=\abs{m}$ and $d=n$ to conclude that $\abs{r_1} \leq (1-\abs{m}/n)^{-1}$, that $\abs{r_2} \leq (1+ \mu_2 \abs{m}_1/n (1-\abs{m_1}/n)+\mu-\mu_2)^{-1}$ and that $\abs{r_3} \leq (1-\abs{m_1}/n)^{-1}(1+\mu_2)^{-1}$. Let $\tau_1=\lrr{ 1+\mu_2 \abs{m_1}/n(1-\abs{m_1}/n) +\mu-\mu_2} \tau$ and $\tau_2= \lrr{ 1- \abs{m_1}/n }(1+\mu_2)\tau$ with $\tau=\ed-1$ and note that we can bound the left-hand side in \eqref{eq:bound.rhohat.pos} from above by
\begin{align}
\begin{split} \label{eq:bound.rhohat.1}
& \P  \lrr{ \lrv{\frac{\hat{s}^2}{s^2} - 1}  > \alpha_1  \lrr{ 1- \frac{\abs{m}}{n} } \tau}  \\
& + \P \lrr{ \lrv{ \frac{Y' P_1 Y/s^2}{n} - \frac{\abs{m_1} + \theta'Z' P_1 Z\theta/s^2}{n} } >  \alpha_2  \tau_1}  \\
& + \P \lrr{ \lrv{ \frac{\theta'Z' P_1 Z\theta/s^2}{n} - \mu_2 \frac{\abs{m_1}}{n} - \frac{\tilde{\theta}_1' Z_1' Z_1 \tilde{\theta}_1/s^2}{n}  } >  \alpha_3  \tau_1}.\\
& + \P \lrr{ \lrv{ \frac{\tilde{\theta}_1' Z_1' Z_1 \tilde{\theta}_1/s^2}{n} - (\mu - \mu_2)  } > \alpha_4 \tau_1 }\\
& + \P \lrr{ \lrv{ \frac{Y' M_1 Y/s^2}{n-\abs{m_1}} - \frac{n - \abs{m_1} + \theta'Z' M_1 Z\theta/s^2}{n-\abs{m_1}} } > \alpha_5  \tau_2}  \\
& + \P \lrr{ \lrv{ \frac{\theta'Z' M_1 Z\theta/s^2 }{n-\abs{m_1}} - \mu_2  } > \alpha_6 \tau_2},
\end{split}
\end{align}
where $\alpha_i \in (0,1)$, $i = 1, \ldots, 6$ are such that $\sum_{i=1}^6 \alpha_i =1$ and where $\tilde{\theta}_1=\theta_1 + S_{1,1}^{-1} S_{1,2} \theta_2$. The left-hand side in \eqref{eq:bound.rhohat.neg} is bounded from above by the same upper bound as in \eqref{eq:bound.rhohat.1} with $\tau=1-\emd$. Lemma \ref{lem:dist} shows that $\hat{s}^2/s^2 \sim \chi^2_{n-\abs{m}}/(n - \abs{m})$ and we can bound the first term in \eqref{eq:bound.rhohat.1} using Lemma~\ref{lem:chi2} from above by
\begin{align} \label{eq:bound.rhohat.11}
2 \exp\lrr{ - (n - \abs{m}) \frac{\alpha_1^2(1-\abs{m}/n)^2 \tau^2}{ 4 (\alpha_1(1-\abs{m}/n) \tau +1)} } .
\end{align}
For the remaining terms, consider first the case where $\mu_2>0$ and $\mu-\mu_2>0$. Recall from \eqref{eq:yp1y} in Lemma~\ref{lem:dist} that $Y'P_1 Y/s^2 \Vert Z \sim \chi^2_{\abs{m_1}}(\theta'Z'P_1 Z \theta/s^2)$ and use Lemma~\ref{lem:chi2.balance} to bound the second term in \eqref{eq:bound.rhohat.1} from above by
\begin{align*}
 2 \E \lrs{ \exp \lrr{ - n  \frac{ \alpha_2^2 \tau_1^2  }{4 \left( \alpha_2 \tau_1  +2(\abs{m_1}/n + \theta' Z' P_1 Z \theta/(s^2n) ) \right) } }  }.
\end{align*}
Integrating over the intersection of $\{  \theta' Z' P_1 Z \theta/(s^2n) \leq \alpha_3 \tau_1 + \mu_2 \abs{m_1}/n + \tilde{\theta}_1' Z_1' Z_1 \tilde{\theta}_1/(s^2n) \}$ and $\{ \tilde{\theta}_1' Z_1' Z_1 \tilde{\theta}_1/(s^2n) \leq \alpha_4 \tau_1 + \mu-\mu_2\}$ and its complement, we can bound the term in the preceding display from above by
\begin{align}
\begin{split} \label{eq:bound.rhohat.121} 
& 2  \exp \lrr{ - n  \frac{ \alpha_2^2 \tau_1^2  }{4 \left( (\alpha_2 + 2 \alpha_3 + 2 \alpha_4) \tau_1+2(\abs{m_1}/n + \mu_2 \abs{m_1}/n + \mu-\mu_2 ) \right) } } \\
& + 2  \P \lrr{  \frac{\theta'Z' P_1 Z\theta/s^2}{n} - \mu_2 \frac{\abs{m_1}}{n} - \frac{\tilde{\theta}_1' Z_1' Z_1 \tilde{\theta}_1/s^2}{n}  > \alpha_3 \tau_1}  \\
& + 2  \P \lrr{  \frac{\tilde{\theta}_1' Z_1' Z_1 \tilde{\theta}_1/s^2}{n} - (\mu - \mu_2)   > \alpha_4 \tau_1 }.
\end{split} 
\end{align}
For the first term in the preceding display, use $\tau_1 \leq (1+ \mu_2 \abs{m_1}/n + \mu- \mu_2 ) \tau$ and $\abs{m_1}/n  + \mu_2 \abs{m_1}/n + \mu-\mu_2 \leq 1 + \mu_2 \abs{m_1}/n + \mu-\mu_2$ in the denominator followed by $\tau_1^2/(1 + \mu_2 \abs{m_1}/n +\mu - \mu_2) \geq (1- \abs{m_1}/n)\tau^2$ in the numerator to bound this term from above by
\begin{align} \label{eq:bound.rhohat.124}
& 2  \exp \lrr{ - n  \frac{ \alpha_2^2 (1-\abs{m_1}/n) \tau^2 }{4 \left( (\alpha_2 + 2 \alpha_3 + 2 \alpha_4) \tau +2 \right)} }.
\end{align}
Recall from \eqref{eq:tzp1} in Lemma~\ref{lem:dist} that $\theta'Z' P_1'Z\theta/s^2 \Vert Z_1 \sim \mu_2 \chi^2_{\abs{m_1}}(\tilde{\theta}_1' Z_1' Z_1 \tilde{\theta}_1/(s^2 \mu_2))$ and use Lemma~\ref{lem:chi2.balance} to bound the sum of the third term in \eqref{eq:bound.rhohat.1} and the second term in \eqref{eq:bound.rhohat.121} from above by
\begin{align}  \label{eq:bound.rhohat.126}
4 \E \lrs{ \exp \lrr{ - n \frac{ \alpha_3^2 \tau_1^2/\mu_2}{4( \alpha_3 \tau_1 + 2(\mu_2 \abs{m_1}/n + \tilde{\theta}_1' Z_1' Z_1 \tilde{\theta}_1/(s^2n))) } } }.
\end{align} 
Integrating over the event $\{ \tilde{\theta}_1' Z_1' Z_1 \tilde{\theta}_1/(s^2n) \leq \alpha_4 \tau_1 + \mu-\mu_2\}$ and its complement, using $\mu_2 \abs{m_1}/n + \mu-\mu_2 \leq 1 + \mu_2 \abs{m_1}/n + \mu-\mu_2$ and $\tau_1 \leq (1 + \mu_2 \abs{m_1}/n + \mu-\mu_2) \tau$ in the denominator followed by $\tau_1^2/(\mu_2 (1 + \mu_2 \abs{m_1}/n + \mu-\mu_2) ) \geq  \abs{m_1}/n(1-\abs{m_1}/n)^2 \tau^2$ in the numerator, we can bound the term in the preceding display from above by
\begin{align} \label{eq:bound.rhohat.125}
& 4 \exp \lrr{ - \abs{m_1} \frac{ \alpha_3^2 (1-\abs{m_1}/n)^2\tau^2}{4( (\alpha_3 + 2 \alpha_4) \tau + 2 ) } } + 4 \P \lrr{ \frac{\tilde{\theta}_1' Z_1' Z_1 \tilde{\theta}_1/s^2}{n} - (\mu - \mu_2)   > \alpha_4 \tau_1 }.
\end{align}
By \eqref{eq:tt1z1} in Lemma~\ref{lem:dist} we have that $\tilde{\theta}_1' Z_1' Z_1 \tilde{\theta}_1/s^2 \sim (\mu - \mu_2) \chi^2_n$. Using Lemma~\ref{lem:chi2} together with the fact that $\tau_1/(\mu-\mu_2) \geq \tau$ for the fourth term in \eqref{eq:bound.rhohat.1}, the third term in \eqref{eq:bound.rhohat.121} and the second term in \eqref{eq:bound.rhohat.125} and collecting the terms in \eqref{eq:bound.rhohat.124} and \eqref{eq:bound.rhohat.125}, we can finally bound the sum of the second, third and fourth term in \eqref{eq:bound.rhohat.1} from above by
\begin{align} 
\begin{split} \label{eq:bound.rhohat.122}
& 2  \exp \lrr{ - n  \frac{ \alpha_2^2 (1-\abs{m_1}/n) \tau^2 }{4 \left( (\alpha_2 + 2 \alpha_3 + 2 \alpha_4) \tau +2 \right)} } + 4 \exp \lrr{ - \abs{m_1} \frac{ \alpha_3^2 (1-\abs{m_1}/n)^2\tau^2}{4( (\alpha_3 + 2 \alpha_4) \tau + 2 ) } } \\
& + 8 \exp \lrr{ - n \frac{\alpha_4^2 \tau^2}{4( \alpha_4 \tau + 1)} }.
\end{split}
\end{align}
Recall from Lemma~\ref{lem:dist} that $Y'M_1 Y/s^2 \Vert Z \sim \chi^2_{n-\abs{m_1}}(\theta'Z'M_1Z \theta/s^2)$ and use Lemma~\ref{lem:chi2.balance} followed by integrating separately over the event $\{ \theta'Z'M_1Z \theta/(s^2(n - \abs{m_1})) \leq \alpha_6 \tau_2 + \mu_2 \}$ and its complement, we can bound the sum of the fifth and sixth term in \eqref{eq:bound.rhohat.1} from above by
\begin{align*} 
2 & \exp \lrr{- (n - \abs{m_1}) \frac{ \alpha_5^2 \tau_2^2}{4( (\alpha_5 + 2 \alpha_6)  \tau_2 +2(1+ \mu_2))} } \\ 
& + 2 \P \lrr{  \frac{\theta'Z' M_1 Z\theta/s^2 }{n-\abs{m_1}} -  \mu_2  > \alpha_6 \tau_2 } + \P \lrr{ \lrv{ \frac{\theta'Z' M_1 Z\theta/s^2 }{n-\abs{m_1}} -  \mu_2  } > \alpha_6 \tau_2 }.
\end{align*}
Plugging in the definition for $\tau_2$ and using $1+\mu_2 \geq 1$ for the first term, recalling from \eqref{eq:tzm1} in Lemma~\ref{lem:dist} that $\theta'Z' M_1 Z\theta/s^2 \sim \mu_2 \chi^2_{n-\abs{m_1} }$ and using that $ \alpha_6 \tau_2/\mu_2 \geq \alpha_6 (1-\abs{m_1}/n) \tau$ we can use Lemma~\ref{lem:chi2} to bound the sum in the preceding display from above by
\begin{align} 
\begin{split} \label{eq:bound.rhohat.123}
2 & \exp \lrr{- (n - \abs{m_1}) \frac{ \alpha_5^2 (1-\abs{m_1}/n)^2 \tau^2}{4( (\alpha_5 + 2 \alpha_6)(1-\abs{m_1}/n) \tau +2)} } \\ 
& + 4 \exp \lrr{ - (n - \abs{m_1}) \frac{\alpha_6^2 (1-\abs{m_1}/n)^2 \tau^2}{4(\alpha_6 (1-\abs{m_1}/n) \tau +1) } } .
\end{split}
\end{align}
Collecting the terms in \eqref{eq:bound.rhohat.11}, \eqref{eq:bound.rhohat.122} and \eqref{eq:bound.rhohat.123}, we can bound the sum in \eqref{eq:bound.rhohat.1} in the case where $\mu_2 >0$ and $\mu - \mu_2 > 0$ from above by
\begin{align} 
\begin{split} \label{eq:bound.rhohat.2}
& 2 \exp\lrr{ - (n - \abs{m}) \frac{\alpha_1^2(1-\abs{m}/n)^2 \tau^2}{ 4 (\alpha_1(1-\abs{m}/n) \tau +1)} } \\
& + 2  \exp \lrr{ - n  \frac{ \alpha_2^2 (1-\abs{m_1}/n) \tau^2 }{4 \left( (\alpha_2 + 2 \alpha_3 + 2 \alpha_4) \tau +2 \right)} } + 4 \exp \lrr{ - \abs{m_1} \frac{ \alpha_3^2 (1-\abs{m_1}/n)^2\tau^2}{4( (\alpha_3 + 2 \alpha_4) \tau + 2 ) } }  \\
& + 8 \exp \lrr{ - n \frac{\alpha_4^2 \tau^2}{4( \alpha_4 \tau + 1)} } + 2 \exp \lrr{- (n - \abs{m_1}) \frac{ \alpha_5^2 (1-\abs{m_1}/n)^2 \tau^2}{4( (\alpha_5 + 2 \alpha_6) (1-\abs{m_1}/n) \tau +2)} } \\ 
& + 4 \exp \lrr{ - (n - \abs{m_1}) \frac{\alpha_6^2 (1-\abs{m_1}/n)^2 \tau^2}{4(\alpha_6 (1-\abs{m_1}/n) \tau +1) } } .
\end{split}
\end{align}
In the case where $\mu_2 > 0$ and $ \mu - \mu_2 =0$, we have that $\tilde{\theta}_1=\mathbf{0}$ and that $\theta'Z'P_1 Z \theta/s^2 \Vert Z_1 \sim \mu_2 \chi^2_{\abs{m_1}}$. Hence, the fourth term in \eqref{eq:bound.rhohat.1} equals zero and we can bound the fifth and sixth term in \eqref{eq:bound.rhohat.1} as before. For the sum of the second and the third term in \eqref{eq:bound.rhohat.1}, we can use the bound in \eqref{eq:bound.rhohat.124} and the tail bound of the central chi-square distribution as in Lemma~\ref{lem:chi2} together with the fact that $\tau_1/(\mu_2 \abs{m_1}/n) \geq (1 - \abs{m_1}/n) \tau$ to bound the sum from above by
\begin{align*}
2  \exp \lrr{ - n  \frac{ \alpha_2^2 (1-\abs{m_1}/n) \tau^2 }{4 \left( (\alpha_2 + 2 \alpha_3 + 2 \alpha_4) \tau +2 \right)} } + 4 \exp \lrr{ - \abs{m_1} \frac{\alpha_3^2 (1- \abs{m_1}/n)^2 \tau^2}{4( \alpha_3 (1- \abs{m_1}/n) \tau + 1)} }. 
\end{align*}
This shows that we can use the bound in \eqref{eq:bound.rhohat.2} also in the case where $\mu_2 > 0$ and $ \mu - \mu_2 =0$. In the case where $\mu_2=0$ and $\mu - \mu_2 >0$, we have that $\theta_2=\mathbf{0}$, that $\theta'Z'P_1 Z \theta /s^2= \theta_1'Z_1'Z_1 \theta_1/s^2=\tilde{\theta}_1'Z_1'Z_1 \tilde{\theta}_1/s^2 \sim \mu \chi^2_n$, $\theta'Z'M_1 Z \theta = 0$ and that $Y'M_1Y/s^2 \Vert Z \sim \chi^2_{n - \abs{m_1}}$. Hence, the third and the sixth term in \eqref{eq:bound.rhohat.1} equals 0 and we can bound the sum of the second and fourth term in \eqref{eq:bound.rhohat.1} as before. Noting that $\tau_2=(1-\abs{m_1}/n) \tau$ here, we can use Lemma~\ref{lem:chi2} to bound the fifth term in \eqref{eq:bound.rhohat.1} from above by
\begin{align*}
2 \exp \lrr{ - (n - \abs{m_1}) \frac{\alpha_5^2 (1-\abs{m_1}/n)^2 \tau^2}{4( \alpha_5 (1-\abs{m_1}/n) \tau + 1)}}.
\end{align*}
This shows that we can use the bound in \eqref{eq:bound.rhohat.2} also in the case where $\mu_2 =0$ and $\mu - \mu_2>0$.  In the remaining case where $\mu_2=\mu - \mu_2=0$, we have that $\theta = \mathbf{0}$ as well as $Y'P_1 Y/s^2 \Vert Z \sim \chi^2_{\abs{m_1}}$ and $Y'M_1 Y/s^2 \Vert Z \sim \chi^2_{n - \abs{m_1}}$. The third, fourth and sixth term in \eqref{eq:bound.rhohat.1} equal 0 and we can bound the fifth term using the bound in the preceding display. Bound the second term in \eqref{eq:bound.rhohat.1} using the tail bound of the chi-square distribution as in Lemma~\ref{lem:chi2} from above by
\begin{align*}
2 \exp \lrr{ - n \frac{\alpha_2^2 (n/\abs{m_1}) \tau^2 }{4( \alpha_2 (n/\abs{m_1}) \tau + 1)}}.
\end{align*}
Hence, we can use the bound in \eqref{eq:bound.rhohat.2} also in this case.
Use the sum in \eqref{eq:bound.rhohat.2} with $\tau=\ed-1$ and that $\gamma(\exp(\delta)-1)+1 \leq \exp(\delta)$ for any $\gamma \in (0,1)$ and $(\alpha_i+2 \alpha_j)(\ed-1)+2 \leq (\alpha_i+2 \alpha_j+2 \alpha_k) (\ed-1)+2 \leq 2 \ed$ for any $i \neq j \neq k \in \{ 1, 2, 3, 4, 5, 6\}$ and that $1 \geq 1-\abs{m_1}/n \geq 1- \abs{m}/n$ to bound the left-hand side in \eqref{eq:bound.rhohat.pos} from above by
\begin{align}
\begin{split} \label{eq:bound.rhohat.pos.1}
2 &\exp\lrr{ - n \lrr{ 1 - \frac{\abs{m}}{n} }^3 \frac{\alpha_1^2(\ed-1)^2}{ 4 \ed } } \\
& + 2  \exp \lrr{ - n \lrr{ 1 - \frac{\abs{m}}{n} }^3 \frac{  \alpha_2^2(\ed-1)^2 }{8 \ed } } \\
& + 4 \exp \lrr{ - \abs{m_1} \lrr{ 1 - \frac{\abs{m}}{n} }^3 \frac{ \alpha_3^2 (\ed-1)^2}{8 \ed } } \\
& + 8 \exp \lrr{ - n \lrr{ 1 - \frac{\abs{m}}{n} }^3 \frac{ \alpha_4^2 (\ed-1)^2 }{4 \ed } }\\
& + 2 \exp \lrr{ - n \lrr{ 1 - \frac{\abs{m}}{n} }^3\frac{\alpha_5^2 (\ed-1)^2}{8 \ed } } \\
& + 4 \exp \lrr{ - n \lrr{ 1 - \frac{\abs{m}}{n} }^3 \frac{\alpha_6^2  (\ed-1)^2 }{4 \ed} }.
\end{split}
\end{align}
To balance the terms in the preceding display, we choose $\alpha_1=\alpha_4=\alpha_6=1/(3(1+\sqrt{2}))$ and $\alpha_2=\alpha_3=\alpha_5 = \sqrt{2}/(3(1+\sqrt{2}))$ which finally shows the statement in \eqref{eq:bound.rhohat.pos}.

To show the statement in \eqref{eq:bound.rhohat.neg}, use the bound in \eqref{eq:bound.rhohat.2} with $\tau=1-\emd =(\ed-1)/\ed$  and note that $\gamma(\exp(\delta)-1)+\ed \leq 2 \exp(\delta)$ for any $\gamma \in (0,1)$ and $(\alpha_i+2 \alpha_j)(\ed-1)+2 \ed \leq (\alpha_i+2 \alpha_j+2 \alpha_k) (\ed-1)+2 \ed \leq 4 \ed$ for any $i \neq j \neq k \in \{ 1, 2, 3, 4, 5, 6\}$ and that $1 \geq 1-\abs{m_1}/n \geq 1- \abs{m}/n$. Hence, we can bound the left hand-side in \eqref{eq:bound.rhohat.neg} from above by
\begin{align}
\begin{split} \label{eq:bound.rhohat.neg.1}
2 &\exp\lrr{ - n \lrr{ 1 - \frac{\abs{m}}{n} }^3 \frac{\alpha_1^2(\ed-1)^2}{ 8 \exp(2 \delta) } } \\
&+ 2  \exp \lrr{ - n \lrr{ 1 - \frac{\abs{m}}{n} }^3 \frac{  \alpha_2^2(\ed-1)^2 }{16 \exp(2 \delta)  } } \\
& + 4 \exp \lrr{ - \abs{m_1} \lrr{ 1 - \frac{\abs{m}}{n} }^3 \frac{ \alpha_3^2(\ed-1)^2}{16 \exp(2 \delta)  } } \\
& + 8 \exp \lrr{ - n \lrr{ 1 - \frac{\abs{m}}{n} }^3 \frac{ \alpha_4^2(\ed-1)^2 }{8 \exp(2 \delta)  } }\\
& + 2 \exp \lrr{ - n \lrr{ 1 - \frac{\abs{m}}{n} }^3\frac{\alpha_5^2  (\ed-1)^2}{16 \exp(2 \delta)  } } \\
& + 4 \exp \lrr{ - n \lrr{ 1 - \frac{\abs{m}}{n} }^3 \frac{\alpha_6^2 (\ed-1)^2 }{ 8 \exp(2 \delta) } }.
\end{split}
\end{align}
To balance the terms in the preceding display, we use the same weights as above which finally shows the statement in \eqref{eq:bound.rhohat.neg}.

\end{proof}

Alternatively, we have the following result where the upper bound depends on the quantity $\mu$!

\begin{appxlem} \label{lem:bound.rhohat.m}
For each $\delta>0$, we have
\begin{align} \label{eq:bound.rhohat.m.pos}
\P(\hat{\rho}^2(m)/r > \exp(\delta)) & \leq 22 \exp \lrr{- n \lrr{1- \frac{\abs{m}}{n}}^3 \frac{(\ed-1)^2}{210 \ed (1+\mu)} }  \\
\intertext{as well as} \label{eq:bound.rhohat.m.neg}
\P( \hat{\rho}^2(m)/r < \exp(-\delta)) & \leq 22 \exp \lrr{- n \lrr{1- \frac{\abs{m}}{n}}^3 \frac{(\ed-1)^2}{840 \exp(2 \delta)(1+\mu)} },
\end{align}
where 
\begin{align}
\begin{split} \label{eq:r.m}
r =  & s^2 \Bigg( (1-a_2)^2 \frac{\abs{m}}{n - \abs{m} +1} +(a_2-a_1)(2-a_1-a_2) \frac{\abs{m_1}}{n - \abs{m_1} +1} + 1 \\
& + a_1^2 ( \mu - \mu_2)  +  a_2^2  \mu_2 + (a_1 - a_2)^2 \frac{\abs{m_1}}{n - \abs{m_1}+1} \mu_2  \Bigg),
\end{split}
\end{align}
where $a_1$ and $a_2$ are the shrinkage factors as in \eqref{eq:a1a2} and where $\mu=\theta'S\theta/ s^2$ and $\mu_2 = \theta_2'(S_{2,2} - S_{2,1} S_{1,1}^{-1} S_{1,2})\theta_2/s^2$.
\end{appxlem}

\begin{proof}
Use the bound in \eqref{eq:bound.rhohat.2} in the proof of Lemma \ref{lem:bound.rhohat} except for the third term. This term originates from bounding the sum
\begin{align} 
\begin{split} \label{eq:bound.rhohat.m.1}
& \P \lrr{  \lrv{\frac{\theta'Z' P_1 Z\theta/s^2}{n} - \mu_2 \frac{\abs{m_1}}{n} - \frac{\tilde{\theta}_1' Z_1' Z_1 \tilde{\theta}_1/s^2}{n} }  > \alpha_3 \tau_1} \\
&+ 2  \P \lrr{  \frac{\theta'Z' P_1 Z\theta/s^2}{n} - \mu_2 \frac{\abs{m_1}}{n} - \frac{\tilde{\theta}_1' Z_1' Z_1 \tilde{\theta}_1/s^2}{n}  > \alpha_3 \tau_1}.
\end{split}
\end{align}
In the case where $\mu_2 >0$ and $\mu - \mu_2 >0$, bound the sum in the preceding display from above by using the bound in \eqref{eq:bound.rhohat.126} followed by integrating over the event $\{\tilde{\theta}_1' Z_1' Z_1 \tilde{\theta}_1/(s^2n) \leq \alpha_4 \tau_1 + \mu - \mu_2 \}$ and its complement to get
\begin{align*}
& 4 \exp \lrr{ - n \frac{\alpha_3^2 \tau_1^2/\mu_2}{4( (\alpha_3 + 2 \alpha_4) \tau_1 + 2( \mu_2 \abs{m_1}/n + \mu - \mu_2) ) }} \\
& + 4 \P \lrr{ \frac{\tilde{\theta}_1' Z_1' Z_1 \tilde{\theta}_1/s^2 }{n} - (\mu - \mu_2) > \alpha_4 \tau_1 }.
\end{align*}
Bound the second term in the preceding display as in the proof of Lemma~\ref{lem:bound.rhohat} (cf. the bound in \eqref{eq:bound.rhohat.125}). For the first term use $\tau_1 \leq (1 + \mu_2 \abs{m_1}/n +\mu-\mu_2) \tau$ and $\mu_2 \abs{m_1}/n + \mu - \mu_2 \leq 1+ \mu_2 \abs{m_1}/n + \mu - \mu_2$ in the denominator followed by $\tau_1^2/(\mu_2(1 + \mu_2 \abs{m_1}/n + \mu-\mu_2 )) \geq \tau^2 (1-\abs{m_1}/n)/(1+ \mu_2) $ in the numerator to bound it from above by
\begin{align} \label{eq:bound.rhohat.m.2}
4 \exp \lrr{ - n \frac{\alpha_3^2 (1- \abs{m_1}/n) \tau^2}{4( (\alpha_3 + 2 \alpha_4) \tau + 2)(1+\mu_2) }}.
\end{align}
Bound the sum in \eqref{eq:bound.rhohat.m.1} in the case where $\mu_2>0$ and $\mu - \mu_2=0$ recalling that $\tilde{\theta}_1=\mathbf{0}$ and that $\theta' Z' P_1 Z \theta /s^2 \Vert Z_1 \sim \mu_2 \chi^2_{\abs{m_1}}$ and using Lemma~\ref{lem:chi2} from above by
\begin{gather*}
4 \exp \lrr{ - n \frac{\alpha_3  \tau_1^2/\mu_2 }{4(\alpha_3 \tau_1 +  \mu_2 \abs{m_1}/n)} }.
\end{gather*}
Using $\alpha_3 \tau_1 +  \mu_2 \abs{m_1}/n \leq (\alpha_3 \tau + 1)(1+ \mu_2 \abs{m_1}/n) \leq ((\alpha_3 + 2 \alpha_4) \tau + 2)(1+ \mu_2 \abs{m_1}/n)$ in the denominator followed by $\tau_1^2/(\mu_2 (1+ \mu_2 \abs{m_1}/n)) \geq \tau^2 (1-\abs{m_1}/n)/(1+\mu_2)$ in the numerator, we can bound the term in the preceding display from above by the term in \eqref{eq:bound.rhohat.m.2}. In the case where $\mu_2=0$, the sum in \eqref{eq:bound.rhohat.m.1} equals 0 and can trivially by bounded from above by the term in \eqref{eq:bound.rhohat.m.2}.

To show the statements in \eqref{eq:bound.rhohat.m.pos} and \eqref{eq:bound.rhohat.m.neg}, we can use the bound in \eqref{eq:bound.rhohat.2} where we replace the third term by the term in \eqref{eq:bound.rhohat.m.2}. Setting $\tau = \ed - 1$ and using $(\alpha_3+2 \alpha_4) (\ed-1) + 2 \leq 2 \ed$ and $1- \abs{m_1}/n \geq (1- \abs{m}/n)^3$ for the term in \eqref{eq:bound.rhohat.m.2}, we can bound the left-hand side in \eqref{eq:bound.rhohat.m.pos} from above by the sum in \eqref{eq:bound.rhohat.2} where we replace the third term by
\begin{align*}
4 \exp \lrr{ - n \lrr{1- \frac{\abs{m}}{n}}^3 \frac{\alpha_3^2 (\ed-1)^2}{8 \ed (1+\mu_2) }}.
\end{align*}
Choosing the weights as in the proof of Lemma \ref{lem:bound.rhohat} and using the fact that $1 \leq 1+\mu_2 \leq 1+\mu$ gives the result. To bound the left-hand side in \eqref{eq:bound.rhohat.m.neg}, set $\tau =1-\emd =(\ed-1)/\ed$ and use $(\alpha_3 + 2 \alpha_4)(\ed-1) + 2 \ed \leq 4 \ed$ and $1- \abs{m_1}/n \geq (1- \abs{m}/n)^3$ in \eqref{eq:bound.rhohat.m.2}. Hence, we can use the upper bound in \eqref{eq:bound.rhohat.neg.1} where we replace the third term by 
\begin{align*}
4 \exp \lrr{ - n \lrr{1- \frac{\abs{m}}{n}}^3 \frac{\alpha_3^2 (\ed-1)^2}{16 \ed (1+\mu_2) }}.
\end{align*}
Using the same weights as in the proof of Lemma~\ref{lem:bound.rhohat} together with the fact that $1 \leq 1+\mu_2 \leq 1+\mu$ gives the result.
\end{proof}

\begin{appxlem} \label{lem:bound.rho}
For each $\delta>0$, we have
\begin{align} \label{eq:bound.rho.pos}
\P(\rho^2(m)/r > \exp(\delta))  & \leq (58+ 2\abs{m_1} ) \exp \lrr{ - \abs{m_1} \frac{\abs{m_1}}{n} \lrr{1 - \frac{\abs{m}}{n}} ^5 \frac{(\ed-1)^2}{10477\exp(2 \delta)} } 
\intertext{as well as} \label{eq:bound.rho.neg}
\P( \rho^2(m)/r < \exp(-\delta)) & \leq (58 + 2 \abs{m_1}) \exp \lrr{ - \abs{m_1} \frac{\abs{m_1}}{n} \lrr{1 - \frac{\abs{m}}{n}} ^5 \frac{(\ed-1)^2}{13089\exp(2 \delta)} } ,
\end{align}
where $r$ was defined in \eqref{eq:r} in Lemma~\ref{lem:bound.rhohat}. 
\end{appxlem}
\begin{proof}
Recall the expansion of the true prediction error
\begin{align}
\begin{split} \label{eq:bound.rho.rho}
\rho^2(m) & = (1-a_2)^2(\hat{\theta}-\theta)'S(\hat{\theta} - \theta) \\
& \phantom{=} + (a_2 - a_1)(2 - a_1-a_2)(\hat{\theta}_1^* - \theta_1^*)'S_{1,1} (\hat{\theta}_1^* - \theta_1^*)   \\
&  \phantom{=} + (a_2-a_1)^2  \theta_2' \tilde{Z}_2'Z_1(Z_1'Z_1)^{-1}  S_{1,1} (Z_1'Z_1)^{-1} Z_1' \tilde{Z}_2 \theta_2\\
& \phantom{=}  + 2 a_1(a_1-a_2) (S_{1,1} \theta_1+S_{1,2} \theta_2)'(Z_1'Z_1)^{-1} Z_1'\tilde{Z}_2 \theta_2  \\
& \phantom{=} + 2a_1(a_1 - 1) (\hat{\theta}_1^* - \theta_1^*)'(S_{1,1} \theta_1 + S_{1,2} \theta_2 ) \\
& \phantom{=} + 2(1-a_1)(a_2-a_1) (\hat{\theta}_1^* - \theta_1^*)'S_{1,1} (Z_1'Z_1)^{-1} Z_1' \tilde{Z}_2 \theta_2 \\
& \phantom{=} +2a_2(a_2-1) (\hat{\theta}_2 - \theta_2)'(S_{2,2} -S_{2,1} S_{1,1}^{-1} S_{1,2})\theta_2\\
& \phantom{=}+ 2 a_1(1-a_2)  (\hat{\theta}_2 - \theta_2)' \tilde{Z}_2'Z_1(Z_1'Z_1)^{-1} (S_{1,1} \theta_1 + S_{1,2} \theta_2 ) \\
& \phantom{=} + 2(1-a_2)(a_1-a_2) (\hat{\theta}_2-\theta_2)' \tilde{Z}_2'Z_1 (Z_1'Z_1)^{-1} S_{1,1}(Z_1'Z_1)^{-1} Z_1'\tilde{Z}_2 \theta_2 \\
& \phantom{=} + 2(1-a_2)(a_1 - a_2) (\hat{\theta}_1^* - \theta_1^*)'S_{1,1} (Z_1'Z_1)^{-1} Z_1' \tilde{Z}_2 (\hat{\theta}_2 - \theta_2)\\
& \phantom{=}+  s^2 + a_1^2  (\theta'S\theta - \theta_2'(S_{2,2} - S_{2,1} S_{1,1}^{-1} S_{1,2}) \theta_2) + a_2^2 \theta_2'(S_{2,2} - S_{2,1} S_{1,1}^{-1} S_{1,2}) \theta_2,
\end{split}
\end{align}
where $a_1$ and $a_2$ are the shrinkage factors as in \eqref{eq:a1a2}. As before let $\mu=\theta'S\theta/s^2$ and $\mu_2=\theta_2(S_{2,2} - S_{2,1} S_{1,1}^{-1} S_{1,2})\theta_2/s^2$ and recall that $\mu - \mu_2= \tilde{\theta}_1'S_{1,1}\tilde{\theta}_1$ with $\tilde{\theta}_1= \theta_1+S_{1,1}^{-1} S_{1,2} \theta_2$. Rewrite the sum of the first three lines and the last line in the previous display as
\begin{align}
\begin{split} \label{eq:bound.rho.1}
 & (1-a_2)^2 \lrs{ (\hat{\theta}-\theta)'S(\hat{\theta} - \theta) - s^2 \frac{\abs{m}}{n - \abs{m} +1} }\\
&+ (a_2 - a_1)(2 - a_1 - a_2) \lrs{ (\hat{\theta}_1^* - \theta_1^*)'S_{1,1} (\hat{\theta}_1^* - \theta_1^*) - s^2 \frac{\abs{m_1}}{ n - \abs{m_1} + 1} }\\
& + (a_2-a_1)^2 \lrs{ \theta_2' \tilde{Z}_2' Z_1(Z_1'Z_1)^{-1}  S_{1,1} (Z_1'Z_1)^{-1} Z_1' \tilde{Z}_2 \theta_2  - s^2 \mu_2  \trace \lrr{ S_{1,1} (Z_1'Z_1)^{-1} } } \\
& + (a_2-a_1)^2  \lrs{s^2 \mu_2 \trace \lrr{S_{1,1}(Z_1'Z_1)^{-1}} - s^2 \mu_2 \frac{\abs{m_1}}{ n - \abs{m_1} +1} } + r.
\end{split}
\end{align}
Hence, the term $\rho^2(m)/r$ is of the form
\begin{align*}
\sum_{i=i}^{11} r_i T_i + 1,
\end{align*}
where the $r_i$'s involve the terms $a_1$, $a_2$ and $r$, e.g., $r_1=(1 - a_2)^2/r$, where $T_1, \ldots, T_4$ are the random variables in squared brackets in display \eqref{eq:bound.rho.1} and $T_5, \ldots, T_{11}$ are the random variables in the fourth up to the tenth line in \eqref{eq:bound.rho.rho}, i.e., the part involving $\theta$, $Z$ and $S$, e.g., $T_5=(S_{1,1} \theta_1+S_{1,2} \theta_2)'(Z_1'Z_1)^{-1} Z_1'\tilde{Z}_2 \theta_2$. Using this notation and setting $\tau=\ed-1$, we can bound the left-hand side in \eqref{eq:bound.rho.pos} from above by
\begin{align}
\label{eq:bound.rho.2}
\sum_{i=i}^{11} \P \lrr{ \abs{r_i} \abs{ T_i} > \alpha_i \tau},
\end{align}
where $\alpha_i \in (0,1)$, $i =1, \ldots, 11$ with $\sum_{i=1}^{11} \alpha_i =1$. The left-hand side in \eqref{eq:bound.rho.neg} is bounded from above by the same upper bound with $\tau=1-\emd$. Note that 
\begin{align}
\begin{split} \label{eq:bound.r}
r & \geq s^2 \lrr{ 1+a_1^2 (\mu-\mu_2) + a_2^2 \mu_2 + (a_1-a_2)^2 \frac{\abs{m_1}}{n - \abs{m_1}+1} \mu_2 } \\
&  \geq s^2 \lrr{ 1+(a_1-a_2)^2 \frac{\abs{m_1}}{n - \abs{m_1}+1} \mu_2 }\geq s^2
\end{split}
\end{align}
Using the third inequality for $r$ and recalling that $a_1 \in [0,1]$ and $a_2 \in [0,1]$, we see that $\abs{r_1} \leq 1/s^2$, $\abs{r_2} \leq 1/s^2$ as well as $\abs{r_{11}} \leq 2/s^2$. For $\mu_2 > 0$ and $\mu - \mu_2 >0$, use the second inequality for $r$ to conclude that $ \abs{r_3} = \abs{r_4} \leq (s^2 \mu_2 \abs{m_1}/(n - \abs{m}_1+1))^{-1}$, and use the first inequality for $r$ to conclude that $ \abs{r_5} \leq ( s^2 \sqrt{\mu_2(\mu-\mu_2) \abs{m_1}/(n-\abs{m_1} +1)} )^{-1}$, that $ \abs{r_6} \leq (s^2\sqrt{\mu-\mu_2})^{-1}$, that $ \abs{r_7} \leq (s^2\sqrt{\mu_2 \abs{m_1}/(n - \abs{m_1} +1 )})^{-1}$, that $ \abs{r_8} \leq (s^2 \sqrt{\mu_2} )^{-1}$, that $\abs{r_9} \leq (s^2\sqrt{\mu-\mu_2})^{-1}$ and that $\abs{r_{10}}\leq (s^2\sqrt{\mu_2 \abs{m_1}/(n - \abs{m_1} +1 )})^{-1}$.\footnote{The inequalities can be rewritten to have the form $2x y \leq 1+x^2+y^2+z$ where $z \geq 0$.} In the case where $\mu_2=0$ it follows that $\theta_2=\mathbf{0}$, i.e., the zero vector, which implies that $T_3=T_4=T_5=T_7=T_8=T_{10}=0$. In the case where $\mu - \mu_2 = 0$, it follows that $\tilde{\theta}_1=\mathbf{0}$ as well as $S_{1,1}\tilde{\theta}_1 = S_{1,1} \theta_1 +S_{1,2} \theta_2=\mathbf{0}$ and hence that $T_5=T_6=T_9=0$. In these cases the corresponding terms in \eqref{eq:bound.rho.2} equal 0 and can be trivially bounded from above by the bounds for the case $\mu_2 > 0$ and $\mu - \mu_2 > 0$, respectively, that are derived in the next paragraphs. Hence, from now on we assume that $\mu_2 > 0$ and $\mu - \mu_2>0$. \\
Use Lemma \ref{lem:dist} together with the upper bounds of $\abs{r_1}$ and $\abs{r_2}$ and the Chernoff bounds as in Lemma A.1 and Lemma A.3 (i) in \cite{leeb08} and the bound in Lemma C.8 (i) in \cite{huber13} to bound the sum of the first and the second term in \eqref{eq:bound.rho.2} from above by 
\begin{align} 
\begin{split} \label{eq:bound.rho.21}
2 &\exp \lrr{ - (n - \abs{m} +1)  \frac{ \alpha_1^2 (1-\abs{m}/(n+1))^2 \tau^2}{4( \alpha_1 (1-\abs{m}/(n+1)) \tau +1 )^2} }\\
& + 2 \exp \lrr{ - (n - \abs{m_1}+1)  \frac{ \alpha_2^2(1-\abs{m_1}/(n+1))^2 \tau^2}{4( \alpha_2 (1-\abs{m_1}/(n+1)) \tau +1 )^2} }.
\end{split}
\end{align}
Lemma \ref{lem:dist} shows that $\tilde{Z}_2 \theta_2/(s \sqrt{\mu_2}) \Vert Z_1 \sim N(0, I_n)$. Set $V_1 = Z_1 S_{1,1}^{-1/2}$ and note that $n \lambda_n(Z_1(Z_1'Z_1)^{-1}  S_{1,1} (Z_1'Z_1)^{-1} Z_1' )=\lambda_1^{-1}(V_1'V_1/n)$ and that $\trace(S_{1,1} (Z_1'Z_1)^{-1}) = \trace((V_1'V_1)^{-1})$. Hence, we can use Corollary \ref{cor:bound.quadraticform} conditional on $Z_1$ together with the upper bound on $\abs{r_3}$ to bound the third term in \eqref{eq:bound.rho.2} from above by
\begin{align*}
2 \E \lrs{ \exp \lrr{ - n \frac{\alpha_3^2\abs{m_1}^2/(n - \abs{m_1}+1)^2 \tau^2}{4 \lambda_1^{-1}(V_1'V_1/n)( \alpha_3\abs{m_1}/(n - \abs{m_1} +1)\tau + \lambda_1^{-1}(V_1'V_1/n) )} } } .
\end{align*}
Note that the term in the preceding display is nonincreasing in $\lambda_1(V_1'V_1/n)$ and that $V_1$ is a $n \times \abs{m_1}$ matrix that has i.i.d.\  standard normally distributed entries. Let 
\begin{align*}
A_1 &= \{ \lambda_1(V_1'V_1/n) > \gamma_1^2(1-\sqrt{\abs{m_1}/n})^2\}
\end{align*}
for some $\gamma_1 \in (0,1)$. Integrate the term in the second-to-last display separately over $A_1$ and its complement, note that the integrand is bounded from above by 1 and use Lemma~\ref{lem:eigenvalues} on the complement of $A_1$ to bound the third term in \eqref{eq:bound.rho.2} finally from above by
\begin{align} 
\begin{split} \label{eq:bound.rho.22}
& 2  \exp \lrr{ - \abs{m_1}  \frac{\abs{m_1}}{n} \frac{\alpha_3^2 \gamma_1^4 n^2 /(n - \abs{m_1}+1)^2(1-\sqrt{\abs{m_1}/n})^4 \tau^2}{4 ( \alpha_3 \gamma_1^2 \abs{m_1}/(n - \abs{m_1}+1)(1-\sqrt{\abs{m_1}/n})^2 \tau + 1) } }  \\
& + 2 \exp \lrr{ - n \frac{(1-\gamma_1)^2 (1- \sqrt{\abs{m_1}/n})^2}{2} }.
\end{split}
\end{align}
Bound the fourth term in \eqref{eq:bound.rho.2} noting that $\trace(S_{1,1} (Z_1'Z_1)^{-1}) = \trace((V_1'V_1)^{-1})$ and using \eqref{eq:bound.trace.rel} in Lemma \ref{lem:bound.trace} together with the upper bound of $\abs{r_4}$ from above by
\begin{align*}
2  \abs{m_1} \exp \lrr{ - (n - \abs{m_1}) \frac{\alpha_4^2\tau^2 }{8( \alpha_4\tau + 1)^2 } }.
\end{align*}
Use \eqref{eq:nd1} in Lemma \ref{lem:dist} to show that $T_5 \Vert Z_1 \sim N(0, s^2 \mu_2 (S_{1,1} \theta_1+S_{1,2} \theta_2)'(Z_1'Z_1)^{-1} (S_{1,1} \theta_1+S_{1,2} \theta_2))$. Use Lemma~\ref{lem:nd}, conditional on $Z_1$, together with the upper bound of $\abs{r_5}$ to bound the fifth term in \eqref{eq:bound.rho.2} from above by
\begin{gather*}
2 \E \lrs{ \exp \lrr{ - \abs{m_1}  \frac{ \alpha_5^2 s^2 (\mu-\mu_2) n/(n - \abs{m_1}+1) \tau^2 }{2 (S_{1,1} \theta_1+S_{1,2} \theta_2)'(Z_1'Z_1/n)^{-1} (S_{1,1} \theta_1+S_{1,2} \theta_2) }} }.
\end{gather*}
We can bound the term in the denominator in the preceding display from above by $ 2s^2 (\mu - \mu_2)/\lambda_1(V_1'V_1/n)$. Using additionally that $n/(n-\abs{m_1}+1) \geq 1$, we can bound the term in the preceding display from above by
\begin{gather*}
2 \E \lrs{ \exp \lrr{ - \abs{m_1} \frac{ \alpha_5^2 \lambda_1(V_1'V_1/n) \tau^2}{2} }}.
\end{gather*}
Because also the terms $T_6$, $T_7$, $T_8$, $T_9$ and $T_{10}$ follow, conditional on $Z$, a centered normal distribution, we will use the same arguments as for the fifth term, mutatis mutandis, for bounding the summands involving these terms. Define $\tilde{V}_2=\tilde{Z}_2 (S_{2,2} - S_{2,1} S_{1,1}^{-1} S_{1,2})^{-1/2}$ and note that $\tilde{V}_2$ is a $n \times \abs{m_2}$-matrix that has i.i.d.\ entries that follow a standard normal distribution. Note that the conditional variance of $T_6$ is bounded from above by $s^4 (\mu - \mu_2)/\lambda_1(V_1'V_1)$, that the conditional variance of $T_7$ is bounded from above by $s^4 \mu_2 \lambda_{\abs{m_2}}(\tilde{V}_2'\tilde{V}_2)/ \lambda_1^2(V_1'V_1)$, that the conditional variance of $T_8$ is bounded from above by $s^4 \mu_2/\lambda_1(\tilde{V}_2'M_1\tilde{V}_2)$, that the conditional variance of $T_9$ is bounded from above by $s^4 (\mu - \mu_2) \lambda_{\abs{m_2}}(\tilde{V}_2' \tilde{V}_2)/(\lambda_1(\tilde{V}_2'M_1\tilde{V}_2) \lambda_1(V_1'V_1))$ and that the conditional variance of $T_{10}$ is bounded from above by $s^4 \mu_2 \lambda^2_{\abs{m_2}}(\tilde{V}_2' \tilde{V}_2)/(\lambda_1(\tilde{V}_2'M_1\tilde{V}_2) \lambda^2_1(V_1'V_1))$. Using the upper bounds of the variances, the upper bounds on $\abs{r_i}$ together with Lemma~\ref{lem:nd} conditional on $Z$ and the facts that $n/(n-\abs{m_1}+1) \geq 1$ as well as $(n-\abs{m_1})/(n-\abs{m_1}+1) \geq 1/2$, we can bound the sum of the sixth up to the tenth term in \eqref{eq:bound.rho.2} from above by
\begin{align*}
& 2 \E \lrs{ \exp \lrr{ - n \frac{ \alpha_6^2 \lambda_1(V_1'V_1/n) \tau^2}{2} }}\\
& + 2 \E \lrs{ \exp \lrr{ - \abs{m_1}  \frac{ \alpha_7^2 \lambda_1^2(V_1'V_1/n) \tau^2}{2 \lambda_{\abs{m_2}}(\tilde{V}_2'\tilde{V}_2/n) } }} \\
& + 2 \E \lrs{ \exp \lrr{ - (n - \abs{m_1}) \frac{ \alpha_8^2 \lambda_1(\tilde{V}_2'M_1\tilde{V}_2/(n - \abs{m_1})) \tau^2}{2} }} \\
& + 2 \E \lrs{ \exp \lrr{ - (n - \abs{m_1}) \frac{ \alpha_9^2 \lambda_1(\tilde{V}_2'M_1\tilde{V}_2/(n - \abs{m_1}))  \lambda_1(V_1'V_1/n) \tau^2}{2 \lambda_{\abs{m_2}}(\tilde{V}_2'\tilde{V}_2/n)}} } \\
& + 2 \E \lrs{ \exp \lrr{ - \abs{m_1} \frac{ \alpha_{10}^2\lambda_1(\tilde{V}_2'M_1\tilde{V}_2/(n - \abs{m_1})) \lambda^2_1(V_1'V_1/n) \tau^2}{4 \lambda^2_{\abs{m_2}}(\tilde{V}_2'\tilde{V}_2/n)} }}.
\end{align*}
In order to get upper bounds for the sum of the previous two displays, we need upper bounds on $\lambda_{\abs{m_2} }(\tilde{V}_2'\tilde{V}_2/n)$ and lower bounds on $\lambda_1(V_1'V_1/n)$ and $\lambda_1(\tilde{V}_2' M_2 \tilde{V}_2/(n - \abs{m_1}))$. Recall $A_1$ and define
\begin{align*}
A_2 &= \{ \lambda_1( \tilde{V}_2'M_1' \tilde{V}_2/(n-\abs{m_1})) > \gamma_2^2(1-\sqrt{\abs{m_2}/(n-\abs{m_1})})^2\} \\
A_3 &= \{ \lambda_{\abs{m_2}}(\tilde{V}_2' \tilde{V}_2/n) < (1+\gamma_3)^2(1 + \sqrt{\abs{m_2}/n})^2\}
\end{align*}
for $\gamma_2 \in (0,1)$ and $\gamma_3>0$. Integrating the corresponding terms separately over $A_1, A_2, A_3$ and its complements, noting that the integrand is bounded from above by 1 and using Lemma~\ref{lem:eigenvalues}, we can bound the sum of the fifth up to the tenth term in \eqref{eq:bound.rho.2} finally from above by
\begin{align}
\begin{split} \label{eq:bound.rho.23}
& 2 \exp \lrr{ - \abs{m_1} \frac{ \alpha_5^2 \gamma_1^2 (1-\sqrt{\abs{m_1}/n})^2 \tau^2 }{2} }\\
& + 2 \exp \lrr{ - n \frac{ \alpha_6^2\gamma_1^2 \ (1-\sqrt{\abs{m_1}/n})^2 \tau^2}{2} }\\
& + 2 \exp \lrr{ - \abs{m_1} \frac{ \alpha_7^2 \gamma_1^4 (1-\sqrt{\abs{m_1}/n})^4 \tau^2}{2 (1+\gamma_3)^2 (1+\sqrt{\abs{m_2}/n})^2 } } \\
& + 2 \exp \lrr{ - (n - \abs{m_1}) \frac{ \alpha_8^2 \gamma_2^2 (1-\sqrt{\abs{m_2}/(n-\abs{m_1})})^2 \tau^2}{2} } \\
& + 2 \exp \lrr{ - (n - \abs{m_1}) \frac{ \alpha_9^2 \gamma_1^2 \gamma_2^2   (1-\sqrt{\abs{m_2}/(n-\abs{m_1})})^2 (1- \sqrt{\abs{m_1}/n})^2 \tau^2 }{2 (1+\gamma_3)^2 (1+\sqrt{\abs{m_2}/n})^2}}  \\
& + 2 \exp \lrr{ - \abs{m_1} \frac{ \alpha_{10}^2 \gamma_1^4 \gamma_2^2 (1-\sqrt{\abs{m_2}/(n-\abs{m_1})})^2 (1- \sqrt{\abs{m_1}/n})^4 \tau^2 }{4(1+\gamma_3)^4 (1+\sqrt{\abs{m_2}/n})^4 } } \\
& + 10 \exp \lrr{ - n  \frac{( 1- \gamma_1)^2 (1-\sqrt{\abs{m_1}/n })^2}{2} } + 6 \exp \lrr{ - n  \frac{\gamma_3^2 (1+\sqrt{\abs{m_2}/n})^2}{2}  } \\
&+ 6 \exp \lrr{ - n  \frac{ (1 - \gamma_2)^2 (1-\sqrt{\abs{m_2}/(n - \abs{m_1})})^2}{2}  }.
\end{split}
\end{align}
The term $T_{11}$ can be rewritten as $w'Qw$ where $Q = Z_1(Z_1'Z_1)^{-1} S_{1,1} (Z_1'Z_1)^{-1} Z_1' \tilde{Z}_2(Z_2'M_1Z_2)^{-1} Z_2'M_1$ and where $w \sim N(0, s^2 I_n)$ as in \eqref{eq:model.z}. Noting that $\trace(Q)=0$ and that $Q Q = \mathbf{0}$, i.e., the zero matrix, we can use Corollary \ref{cor:qf.traceless0}, conditional on $Z$, together with the upper bound on $\abs{r_{11}}$ to bound the last term in \eqref{eq:bound.rho.2} if $\lambda_n(Q + Q')>0$ from above by 
\begin{gather} \label{eq:bound.rho.241}
4 \E \lrs{ \exp \lrr{ - n \frac{  \alpha_{11}^2 \tau^2}{8 n \sqrt{2 \lambda_n(Q'Q)} ( \alpha_{11} \tau + 2 n \sqrt{2 \lambda_n(Q'Q)})} } }.
\end{gather}
Recall from the proof of Lemma~\ref{lem:qf.traceless0} that $\lambda_n(Q'Q) >0$ in this case. In the case where $\lambda_n(Q+Q')=0$, the corresponding upper bound equals 0. Note that $M_1 Z_2 = M_1 \tilde{Z}_2$, that 
\begin{gather*}
Q'Q = M_1 \tilde{V}_2 (\tilde{V}_2'M_1\tilde{V}_2)^{-1} \tilde{V}_2' V_1 (V_1'V_1)^{-3} V_1'\tilde{V}_2 (\tilde{V}_2'M_1 \tilde{V}_2)^{-1} \tilde{V}_2'M_1,
\end{gather*}
that the term in \eqref{eq:bound.rho.241} is nondecreasing in $\lambda_n(Q'Q)$ and finally that
\begin{align*}
\lambda_n(Q'Q) & \leq \frac{1}{n(n-\abs{m_1})} \frac{\lambda_{\abs{m_2}}(\tilde{V}_2'\tilde{V}_2/n ) }{ \lambda_1( \tilde{V}_2'M_1\tilde{V}_2 /(n-\abs{m_1})) \lambda_1^2(V_1'V_1/n)}.
\end{align*}
Integrating the term in \eqref{eq:bound.rho.241} separately over $A_1 \cap A_2 \cap A_3$ and its complement, noting that the integrand is bounded from above by 1 and using Lemma~\ref{lem:eigenvalues}, we can bound the last term in \eqref{eq:bound.rho.2} from above by
\begin{align}
\begin{split} \label{eq:bound.rho.24}
& + 4  \exp \lrr{ - n \frac{  \alpha_{11}^2 \gamma_1^4 \gamma_2^2  C_1^2 \tau^2 }{8  \sqrt{2 } (1+\gamma_3)  \left(  \alpha_{11} \gamma_1^2 \gamma_2  C_1 \tau + 2  \sqrt{2 } (1+\gamma_3)  \right)  }  } \\
& + 4 \exp \lrr{ - n  \frac{( 1- \gamma_1)^2 (1-\sqrt{\abs{m_1}/n })^2}{2} } + 4 \exp \lrr{ - n  \frac{\gamma_3^2 (1+\sqrt{\abs{m_2}/n})^2}{2}  } \\
&+ 4 \exp \lrr{ - n  \frac{ (1 - \gamma_2)^2 (1-\sqrt{\abs{m_2}/(n - \abs{m_1})})^2}{2}  },
\end{split}
\end{align}
where $C_1 = (1-\sqrt{\abs{m_2}/(n - \abs{m_1})})( 1- \sqrt{ \abs{m_1}/n})^2 \sqrt{1-\abs{m_1}/n}/(1+\sqrt{\abs{m_2}/n})$.

Collecting the terms in \eqref{eq:bound.rho.21}, \eqref{eq:bound.rho.22}, \eqref{eq:bound.rho.23} and \eqref{eq:bound.rho.24} and using the facts that $(n+1)(1-\abs{m_1}/(n+1))^3 \geq (n+1)(1-\abs{m}/(n+1))^3 \geq n(1-\abs{m}/n)^3$, $n/(n- \abs{m_1}+1) \geq 1$, $\abs{m_1}/(n- \abs{m_1} + 1) (1 - \sqrt{\abs{m_1}/n})^2 \leq 1$ and that $1 \leq 1+ \sqrt{\abs{m_2}/n} \leq 2$, we can bound the sum in \eqref{eq:bound.rho.2} finally from above by
\begin{align}
\begin{split} \label{eq:bound.rho.3}
2 &\exp \lrr{ - n \lrr{1 - \frac{\abs{m}}{n} }^3   \frac{ \alpha_1^2 \tau^2}{4( \alpha_1(1-\abs{m}/(n+1)) \tau +1 )^2 } }\\
& + 2 \exp \lrr{ - n \lrr{1 - \frac{\abs{m}}{n} }^3    \frac{ \alpha_2^2  \tau^2 }{4( \alpha_2(1-\abs{m_1}/(n+1)) \tau +1 )^2} } \\
& + 2  \exp \lrr{ - \abs{m_1} \frac{\abs{m_1}}{n}   \frac{\alpha_3^2 \gamma_1^4 (1-\sqrt{\abs{m_1}/n})^4 \tau^2}{4 ( \alpha_3 \gamma_1^2 \tau + 1) } }  \\
& + 2 \abs{m_1} \exp \lrr{ - n \lrr{1 - \frac{\abs{m_1}}{n} } \frac{\alpha_4^2 \tau^2 }{8( \alpha_4 \tau + 1)^2 } } \\
& + 2 \exp \lrr{ - \abs{m_1} \frac{ \alpha_5^2 \gamma_1^2 (1-\sqrt{\abs{m_1}/n})^2 \tau^2 }{2} }\\
& + 2 \exp \lrr{ - n \frac{ \alpha_6^2\gamma_1^2 \ (1-\sqrt{\abs{m_1}/n})^2 \tau^2}{2} }\\
& + 2 \exp \lrr{ - \abs{m_1} \frac{ \alpha_7^2 \gamma_1^4 (1-\sqrt{\abs{m_1}/n})^4 \tau^2}{8 (1+\gamma_3)^2  } } \\
& + 2 \exp \lrr{ - n \lrr{1 - \frac{\abs{m_1}}{n} } \frac{ \alpha_8^2 \gamma_2^2 (1-\sqrt{\abs{m_2}/(n-\abs{m_1})})^2 \tau^2}{2} } \\
& + 2 \exp \lrr{ - n \lrr{1 - \frac{\abs{m_1}}{n} } \frac{ \alpha_9^2 \gamma_1^2 \gamma_2^2   (1-\sqrt{\abs{m_2}/(n-\abs{m_1})})^2 (1- \sqrt{\abs{m_1}/n})^2 \tau^2 }{8 (1+\gamma_3)^2}}  \\
& + 2 \exp \lrr{ - \abs{m_1} \frac{ \alpha_{10}^2 \gamma_1^4 \gamma_2^2 (1-\sqrt{\abs{m_2}/(n-\abs{m_1})})^2 (1- \sqrt{\abs{m_1}/n})^4 \tau^2 }{64(1+\gamma_3)^4  } } \\
& + 4  \exp \lrr{ - n \frac{  \alpha_{11}^2 \gamma_1^4 \gamma_2^2  \tilde{C}_1^2 \tau^2 }{16  \sqrt{2 } (1+\gamma_3)  \left(  \alpha_{11} \gamma_1^2 \gamma_2  \tilde{C}_1 \tau + 4  \sqrt{2 } (1+\gamma_3)  \right)  }  } \\
& + 16 \exp \lrr{ - n  \frac{( 1- \gamma_1)^2 (1-\sqrt{\abs{m_1}/n })^2}{2} } + 10 \exp \lrr{ - n  \frac{\gamma_3^2 }{2}  } \\
&+ 10 \exp \lrr{ - n  \frac{ (1 - \gamma_2)^2 (1-\sqrt{\abs{m_2}/(n - \abs{m_1})})^2}{2}  }.
\end{split}
\end{align}
where $\tilde{C}_1 = (1-\sqrt{\abs{m_2}/(n - \abs{m_1})})( 
1- \sqrt{ \abs{m_1}/n})^2 \sqrt{1-\abs{m_1}/n}$. 
For the statement in \eqref{eq:bound.rho.pos} use the bound in 
the preceding display with $\tau= \ed-1$, note that $\alpha(\ed-1) + 
D \leq D \ed \leq D \exp(2 \delta)$ holds for any $D \geq 1$ and 
$\alpha \in (0,1)$ and that $\ed \geq 1 \geq (\ed-1)^2/\exp(2 \delta)$. 
Using this together with the inequalities in Lemma~\ref{lem:prop}~(ii), 
we can bound the left hand-side in \eqref{eq:bound.rho.pos} from above by
\begin{align}
\begin{split} \label{eq:bound.rho.pos.1}
2 &\exp \lrr{ - n \frac{ \alpha_1^2 }{4 } \psi(\delta, \abs{m}/n)} + 2 \exp \lrr{ - n  \frac{ \alpha_2^2  }{4 } \psi(\delta, \abs{m}/n) } \\
& + 2  \exp \lrr{ - \abs{m_1} \frac{\abs{m_1}}{n}   \frac{\alpha_3^2 \gamma_1^4 }{18  } \psi(\delta, \abs{m}/n)}  + 2 \abs{m_1} \exp \lrr{ - n   \frac{\alpha_4^2}{8  } \psi(\delta, \abs{m}/n)} \\
& + 2 \exp \lrr{ - \abs{m_1}  \frac{3 \alpha_5^2 \gamma_1^2 }{8 } \psi(\delta, \abs{m}/n)} + 2 \exp \lrr{ - n   \frac{3 \alpha_6^2\gamma_1^2  }{8 } \psi(\delta, \abs{m}/n)}\\
& + 2 \exp \lrr{ - \abs{m_1}   \frac{ \alpha_7^2 \gamma_1^4 }{36 (1+\gamma_3)^2  } \psi(\delta, \abs{m}/n)}  + 2 \exp \lrr{ - n   \frac{3 \alpha_8^2 \gamma_2^2  }{8 } \psi(\delta, \abs{m}/n)} \\
& + 2 \exp \lrr{ - n   \frac{ \alpha_9^2 \gamma_1^2 \gamma_2^2   }{16 (1+\gamma_3)^2 } \psi(\delta, \abs{m}/n)}   + 2 \exp \lrr{ - \abs{m_1}   \frac{ \alpha_{10}^2 \gamma_1^4 \gamma_2^2 }{360(1+\gamma_3)^4  } \psi(\delta, \abs{m}/n)} \\
& + 4  \exp \lrr{ - n   \frac{  \alpha_{11}^2 \gamma_1^4 \gamma_2^2  }{2640 (1+\gamma_3)^2   } \psi(\delta, \abs{m}/n) }  + 16 \exp \lrr{ - n    \frac{3 ( 1- \gamma_1)^2  }{8} \psi(\delta, \abs{m}/n)} \\
&+ 10 \exp \lrr{ - n   \frac{\gamma_3^2 }{2}  \psi(\delta, \abs{m}/n)}  + 10 \exp \lrr{ - n    \frac{ 3 (1 - \gamma_2)^2 }{8} \psi(\delta, \abs{m}/n) },
\end{split}
\end{align}
where $\psi(x,y) = (1-y)^5(\exp(x)-1)^2/\exp(2x)$.
To balance the terms in the preceding display, choose the weights to be $\alpha_1 = \alpha_2 =1/51$, $\alpha_3=11/255$, $\alpha_4=43/1530$, $\alpha_5=\alpha_6=\alpha_8=5/306$, $\alpha_7=19/306$, $\alpha_9=21/510$,  $\alpha_{10}=1/5$, $\alpha_{11}=137/255$ and $\gamma_1=\gamma_2=1- \alpha_1 \sqrt{2/3}$, $\gamma_3=\alpha_1/\sqrt{2}$. This shows the statement in \eqref{eq:bound.rho.pos}.

For the second statement, use the bound in \eqref{eq:bound.rho.3} 
with $\tau=1-\emd =(\ed-1)/\ed$, note that $\alpha(\ed-1)+D \ed \leq (D+1) \ed$
for every $D \geq 0$ and every $\alpha \in (0,1)$  and that 
$1 \geq (\ed-1)^2/\exp(2 \delta)$. Using this together with the 
inequalities in Lemma~\ref{lem:prop}~(ii), we can bound the left 
 hand-side in \eqref{eq:bound.rho.neg} from above by
%
\begin{align}
\begin{split} \label{eq:bound.rho.neg.1}
2 &\exp \lrr{ - n \frac{ \alpha_1^2 }{16 } \psi(\delta, \abs{m}/n)} + 2 \exp \lrr{ - n  \frac{ \alpha_2^2  }{16 } \psi(\delta, \abs{m}/n) } \\
& + 2  \exp \lrr{ - \abs{m_1} \frac{\abs{m_1}}{n}   \frac{\alpha_3^2 \gamma_1^4 }{36  } \psi(\delta, \abs{m}/n)}  + 2 \abs{m_1} \exp \lrr{ - n   \frac{\alpha_4^2}{32  } \psi(\delta, \abs{m}/n)} \\
& + 2 \exp \lrr{ - \abs{m_1}  \frac{3 \alpha_5^2 \gamma_1^2 }{8 } \psi(\delta, \abs{m}/n)} + 2 \exp \lrr{ - n   \frac{3 \alpha_6^2\gamma_1^2  }{8 } \psi(\delta, \abs{m}/n)}\\
& + 2 \exp \lrr{ - \abs{m_1}   \frac{ \alpha_7^2 \gamma_1^4 }{36 (1+\gamma_3)^2  } \psi(\delta, \abs{m}/n)}  + 2 \exp \lrr{ - n   \frac{3 \alpha_8^2 \gamma_2^2  }{8 } \psi(\delta, \abs{m}/n)} \\
& + 2 \exp \lrr{ - n   \frac{ \alpha_9^2 \gamma_1^2 \gamma_2^2   }{16 (1+\gamma_3)^2 } \psi(\delta, \abs{m}/n)}   + 2 \exp \lrr{ - \abs{m_1}   \frac{ \alpha_{10}^2 \gamma_1^4 \gamma_2^2 }{360(1+\gamma_3)^4  } \psi(\delta, \abs{m}/n)} \\
& + 4  \exp \lrr{ - n   \frac{  \alpha_{11}^2 \gamma_1^4 \gamma_2^2  }{330\sqrt{2} (1+\gamma_3)(1+4 \sqrt{2}(1+\gamma_3))   } \psi(\delta, \abs{m}/n) }  \\
& + 16 \exp \lrr{ - n    \frac{3 ( 1- \gamma_1)^2  }{8} \psi(\delta, \abs{m}/n)} + 10 \exp \lrr{ - n   \frac{\gamma_3^2 }{2}  \psi(\delta, \abs{m}/n)} \\
&  + 10 \exp \lrr{ - n    \frac{ 3 (1 - \gamma_2)^2 }{8} \psi(\delta, \abs{m}/n) },
\end{split}
\end{align}
To balance the terms in the preceding display, choose the weights to be $\alpha_1 = \alpha_2 = 5/143$, $\alpha_3 = 541/10010$, $\alpha_4=99/2002$, $\alpha_5=\alpha_6=\alpha_8=29/2002$, $\alpha_7 = 547/10010$, $\alpha_9 = 73/2002$, $\alpha_{10} = 1777/10010$, $\alpha_{11}=1030/2002$, $\gamma_1 =\gamma_2 = 1-\alpha_1/\sqrt{6}$ and $\gamma_3=\alpha_1/\sqrt{8}$ which shows the statement in \eqref{eq:bound.rho.neg}.
\end{proof}
\begin{appxlem} \label{lem:bound.rho.m}
For each $\delta>0$, we have
\begin{align} \label{eq:bound.rho.m.pos}
\P(\rho^2(m)/r > \exp(\delta))  & \leq (58 + 2 \abs{m_1}) \exp \lrr{ - n \lrr{1 - \frac{\abs{m}}{n}}^5  \frac{ (\ed-1)^2}{19371 (1+\mu)^2 \exp(2 \delta)} } 
\intertext{as well as} \label{eq:bound.rho.m.neg}
\P( \rho^2(m)/r < \exp(-\delta)) & \leq (58 + 2 \abs{m_1}) \exp \lrr{ - n \lrr{1 - \frac{\abs{m}}{n}}^5  \frac{ (\ed-1)^2}{22534 (1+\mu)^2 \exp(2 \delta)} } ,
\end{align}
where $r$ was defined in \eqref{eq:r} in Lemma~\ref{lem:bound.rhohat}. 
\end{appxlem}
%
\begin{proof}
We take over some of the upper bounds from Lemma \ref{lem:bound.rho} but use different bounds for the third, the fifth up to the eighth and the tenth term in \eqref{eq:bound.rho.2}. As in the proof of Lemma \ref{lem:bound.rho}, it suffices to consider only the case where $\mu_2>0$ and $\mu - \mu_2 >0$. In the case where $\mu_2=0$ we have that $\theta_2=\mathbf{0}$ and hence that all of the above mentioned terms equal 0, in the case where $\mu - \mu_2=0$, we have that $\tilde{\theta}_1= \mathbf{0}$ and hence that the fifth term equals 0. In these cases we can trivially bound the corresponding terms by the upper bounds that we will derive now. Use $\abs{r_3} \leq 1/s^2$ together with Corollary \ref{cor:bound.quadraticform} conditional on $Z$ as before to bound the third term in \eqref{eq:bound.rho.2} from above by
\begin{align*}
2 \E \lrs{ \exp \lrr{ - n \frac{\alpha_3^2 \tau^2 /\mu_2^2}{4 \lambda_1^{-1}(V_1'V_1/n)(\alpha_3 \tau/\mu_2 + \lambda_1^{-1}(V_1'V_1/n) ) } } }.
\end{align*}
As in the proof of Lemma~\ref{lem:bound.rho} let $A_1=\{ \lambda_1(V_1'V_1/n) > \gamma_1^2 (1 - \sqrt{\abs{m_1}/n})^2 \}$ and integrate separately over the event $A_1$ and its complement to bound the term in the preceding display from above by (cf. the bound in \eqref{eq:bound.rho.22})
\begin{align*}
2 & \exp \lrr{ - n \frac{\alpha_3^2 \gamma_1^4 (1 - \sqrt{\abs{m_1}/n})^4 \tau^2 }{4 \mu_2 ( \alpha_3 \gamma_1^2 (1 - \sqrt{\abs{m_1}/n})^2 \tau + \mu_2 ) } } \\
& + 2 \exp \lrr{ - n \frac{(1-\gamma_1)^2  (1 - \sqrt{\abs{m_1}/n})^2}{2}} .
\end{align*}
For the remaining four terms, note that $r \geq s^2 (1+a_1^2(\mu - \mu_2) ) \geq s^2$ which implies that $\abs{r_5} \leq 2 a_1/r \leq (s^2 \sqrt{\mu - \mu_2})^{-1}$, $\abs{r_6} \leq 2 a_1(1 - a_1)/s^2 \leq 1/(2 s^2)$, $\abs{r_7} \leq 2/s^2$, $\abs{r_8} \leq 2 a_2(1-a_2)/s^2 \leq 1/(2s^2)$ and $\abs{r_{10}} \leq 2/s^2$. Using these upper bounds together with the arguments in the proof of Lemma \ref{lem:bound.rho} to bound the sum of the fifth up to the eighth and the tenth term in \eqref{eq:bound.rho.2} from above by
\begin{align*}
2 & \E \lrs{ \exp \lrr{ - n \frac{\alpha_5^2  \lambda_1(V_1'V_1/n) \tau^2 }{ 2 \mu_2} } } \\
& + 2  \E \lrs{ \exp \lrr{ - n \frac{2 \alpha_6^2 \lambda_1(V_1'V_1/n) \tau^2}{ \mu-\mu_2  } } }\\
& + 2  \E \lrs{ \exp \lrr{ - n \frac{\alpha_7^2 \lambda_1^2(V_1'V_1/n) \tau^2}{8 \mu_2 \lambda_{\abs{m_2}}(\tilde{V}_2' \tilde{V}_2/n)} } } \\
& + 2  \E \lrs{ \exp \lrr{ - (n - \abs{m_1}) \frac{2 \alpha_8^2 \lambda_1(\tilde{V}_2'M_1 \tilde{V}_2/(n- \abs{m_1})) \tau^2}{ \mu_2  } } } \\
& + 2  \E \lrs{ \exp \lrr{ - (n - \abs{m_1}) \frac{\alpha_{10}^2 \lambda_1(\tilde{V}_2'M_1 \tilde{V}_2/(n- \abs{m_1})) \lambda_1^2(V_1'V_1/n) \tau^2}{8 \mu_2 \lambda_{\abs{m_2}}^2(\tilde{V}_2' \tilde{V}_2/n)} } } .
\end{align*}
Integrating the corresponding terms over $A_1$, $A_2$, $A_3$ and its complements as in the proof of Lemma~\ref{lem:bound.rho}, mutatis mutandis, we can finally bound the sum of the third, the fifth up to the eighth and the tenth term in \eqref{eq:bound.rho.2} from above by
\begin{align*}
2 & \exp \lrr{ - n \frac{\alpha_3^2 \gamma_1^4 (1 - \sqrt{\abs{m_1}/n})^4 \tau^2 }{4 \mu_2 ( \alpha_3 \gamma_1^2 (1 - \sqrt{\abs{m_1}/n})^2 \tau + \mu_2 ) } } \\
& + 2  \exp \lrr{ - n \frac{\alpha_5^2 \gamma_1^2 (1 - \sqrt{\abs{m_1}/n})^2 \tau^2 }{ 2 \mu_2} }  \\
& + 2  \exp \lrr{ - n \frac{2 \alpha_6^2 \gamma_1^2 (1 - \sqrt{\abs{m_1}/n})^2\tau^2}{ \mu-\mu_2  } } \\
& + 2   \exp \lrr{ - n \frac{\alpha_7^2 \gamma_1^4 (1 - \sqrt{\abs{m_1}/n})^4 \tau^2}{8 \mu_2 (1+\gamma_3)^2(1+ \sqrt{\abs{m_2}/n})^2} }  \\
& + 2  \exp \lrr{ - (n - \abs{m_1}) \frac{2 \alpha_8^2 \gamma_2^2 (1 - \sqrt{\abs{m_2}/(n - \abs{m_1})})^2 \tau^2}{ \mu_2  } }  \\
& + 2  \exp \lrr{ - (n - \abs{m_1}) \frac{\alpha_{10}^2 \gamma_1^4 \gamma_2^2 (1 - \sqrt{\abs{m_2}/(n - \abs{m_1})})^2 (1 - \sqrt{\abs{m_1}/n})^4 \tau^2}{8 \mu_2 (1+\gamma_3)^4 (1+ \sqrt{\abs{m_2}/n})^4} } \\
& + 10 \exp \lrr{ - n \frac{(1-\gamma_1)^2  (1 - \sqrt{\abs{m_1}/n})^2}{2}} +  4 \exp \lrr{ - n \frac{\gamma_3^2 (1+ \sqrt{\abs{m_2}/n})^2}{2}} \\
& + 4 \exp \lrr{ - n \frac{(1-\gamma_2)^2  (1 - \sqrt{\abs{m_2}/(n - \abs{m_1})})^2}{2}}.
\end{align*}
Plugging in $\tau=\ed-1$, using $\mu_2(\alpha (\ed-1) + \mu_2) \leq (1+\mu_2)^2 \ed$ for any $\alpha \in [0,1]$, $\mu_2 \leq 1+ \mu_2$ as well as $1 \leq 1+\sqrt{\abs{m_2}/n} \leq 2$, $1-\abs{m_1}/n \geq 1- \abs{m}/n$ and $\ed \geq 1 \geq (\ed-1)^2/\exp(2 \delta)$ together with the bounds in Lemma~\ref{lem:prop}~(ii), we can bound the sum of the first six terms in the preceding display from above by
\begin{align*}
2 & \exp \lrr{ - n \lrr{1 - \frac{\abs{m_1}}{n}}^5 \frac{\alpha_3^2 \gamma_1^4  (\ed-1)^2 }{18 (1+\mu_2)^2 \exp(2 \delta) } } \\
& + 2  \exp \lrr{ - n  \lrr{1 - \frac{\abs{m_1}}{n}}^5 \frac{3\alpha_5^2 \gamma_1^2 (\ed-1)^2 }{ 8 (1+\mu_2) \exp(2 \delta) } }  \\
& + 2  \exp \lrr{ - n  \lrr{1 - \frac{\abs{m_1}}{n}}^5 \frac{3 \alpha_6^2 \gamma_1^2 (\ed-1)^2}{ 2 (1+\mu-\mu_2) \exp(2 \delta) } } \\
& + 2   \exp \lrr{ - n  \lrr{1 - \frac{\abs{m_1}}{n}}^5\frac{\alpha_7^2 \gamma_1^4 (\ed-1)^2 }{144 (1+\gamma_3)^2 (1+\mu_2) \exp(2 \delta)} }  \\
& + 2  \exp \lrr{ - n  \lrr{1 - \frac{\abs{m_1}}{n}}^5 \frac{3 \alpha_8^2 \gamma_2^2 (\ed-1)^2}{2(1+ \mu_2) \exp(2 \delta)  } }  \\
& + 2  \exp \lrr{ - n  \lrr{1 - \frac{\abs{m}}{n}}^5  \frac{\alpha_{10}^2 \gamma_1^4 \gamma_2^2(\ed-1)^2}{2640 (1+\gamma_3)^4 (1+\mu_2) \exp(2 \delta)} }
\end{align*}
Hence, we can bound the left-hand side in \eqref{eq:bound.rho.m.pos} from above by the sum in \eqref{eq:bound.rho.pos.1} where we replace the sum of the third, fifth up to the eighth and tenth term by the sum in the preceding display. Choosing the weights to be $\alpha_1=\alpha_2 = 7/487$, $\alpha_3= 76/2435$, $\alpha_4 = 99/4870$, $\alpha_5= 29/2435 $, $\alpha_6=\alpha_8 = 29/4870$, $\alpha_7=87/974$, $\alpha_9=29/974$, $\alpha_{10} = 1901/4870$, $\alpha_{11}=941/2435$, $\gamma_1=\gamma_2 = 1-\alpha_1 \sqrt{2/3}$, $\gamma_3=\alpha_1/\sqrt{2}$ and noting that $1 \leq 1+\mu-\mu_2 \leq 1+\mu $, $1 \leq 1+\mu_2 \leq 1+\mu$ as well as $\abs{m_1} (58/\abs{m_1}+2) \leq \abs{m_1} 64/3$ gives the statement.
To bound the left-hand side in \eqref{eq:bound.rho.m.neg}, we use the bound in the second-to-last display with $\tau = 1-\emd =(\ed - 1)/\ed$ and the facts that $\mu_2 \ed (\alpha (\ed-1) + \mu_2 \ed) \leq (1+\mu_2)^2 \exp(2 \delta)$ for any $\alpha \in [0,1]$, $\mu_2 \leq 1+ \mu_2$ as well as $1 \leq 1+\sqrt{\abs{m_2}/n} \leq 2$, $1-\abs{m_1}/n \geq 1- \abs{m}/n$ and $1 \geq (\ed-1)^2/\exp(2 \delta)$ together with the bounds in Lemma~\ref{lem:prop}~(ii) and get the upper bound as in the preceding display. Hence, we can bound the left-hand side in \eqref{eq:bound.rho.m.neg} from above by using the sum in \eqref{eq:bound.rho.neg.1} where we replace the sum of the third, fifth up to the eighth and tenth term by the sum in the preceding display. Choosing the weights to be $\alpha_1=\alpha_2 = 2/75$, $\alpha_3= 13/450$, $\alpha_4 = 17/450$, $\alpha_5= 11/1000$, $\alpha_6=\alpha_8 = 11/2000$, $\alpha_7=743/9000$, $\alpha_9=248/9000$, $\alpha_{10} = 649/1800$, $\alpha_{11}=581/1500$, $\gamma_1=\gamma_2 = 1-\alpha_1/ \sqrt{6}$, $\gamma_3=\alpha_1/\sqrt{8}$ and noting that $1 \leq 1+\mu-\mu_2 \leq 1+\mu $ and that $1 \leq 1+\mu_2 \leq 1+\mu$ gives the statement.

\end{proof}

\begin{proof}[Proof of Theorem \ref{th:help.bound}]

The left-hand side of \eqref{eq:help.bound} can be rewritten as
\begin{align*}
\P\lrr{ \frac{\hat{\rho}^2(m)}{\rho^2(m)} \geq \exp(\eps) } + \P\lrr{ \frac{\hat{\rho}^2(m)}{\rho^2(m)} \leq \exp(- \eps) }.
\end{align*}
We can bound the sum in the preceding display for any $\alpha_1, \alpha_2 \in (0,1)$ from above by
\begin{align*}
\P & \lrr{ \hat{\rho}^2(m)/r \geq \exp( \alpha_1 \eps) } + \P\lrr{ \rho^2(m)/r \leq \exp(- (1 - \alpha_1) \eps) } \\
& + \P\lrr{ \hat{\rho}^2(m)/r \leq \exp(- \alpha_2 \eps) } + \P\lrr{ \rho^2(m)/r \geq \exp( (1-\alpha_2) \eps ) },
\end{align*}
where $r$ was defined in \eqref{eq:r} in Lemma~\ref{lem:bound.rhohat}. Note that $(\exp(\alpha \eps)-1)^2/\exp(\alpha \eps) \geq (\exp(\alpha \eps)-1)^2/\exp(2 \alpha \eps) \geq \alpha^2 \eps^2/(1+ \eps)^2$ holds for any $\alpha \in (0,1)$ and $\eps >0$. Using this fact together with Lemma~\ref{lem:bound.rhohat} and Lemma~\ref{lem:bound.rho}, we can bound the sum in the preceding display from above by
\begin{align*}
22 & \exp \lrr{ - \abs{m_1} \lrr{1 - \frac{\abs{m}}{n} }^3   \frac{ \alpha_1^2 \eps^2}{210 (1+\eps)^2 } } \\
 & + (58 + 2 \abs{m_1}) \exp \lrr{ - \abs{m_1} \frac{\abs{m_1}}{n} \lrr{1 - \frac{\abs{m}}{n} }^5   \frac{ (1-\alpha_1)^2 \eps^2}{13089 (1+ \eps)^2 } } \\
& + 22 \exp \lrr{ - \abs{m_1} \lrr{1 - \frac{\abs{m}}{n} }^3   \frac{ \alpha_2^2 \eps^2}{840 (1+\eps)^2 } } \\
& + (58 + 2 \abs{m_1}) \exp \lrr{ - \abs{m_1} \frac{\abs{m_1}}{n} \lrr{1 - \frac{\abs{m}}{n} }^5   \frac{ (1-\alpha_2)^2 \eps^2}{10477 (1+ \eps)^2 } }.
\end{align*}
Choosing $\alpha_1 \in (0,1)$ such that the exponents of the first two terms in the preceding display coincide gives $\alpha_1 = \sqrt{210} \sqrt{\abs{m_1}/n}(1-\abs{m}/n)/(\sqrt{210}\sqrt{\abs{m_1}/n}(1-\abs{m}/n)+\sqrt{13089})$. Doing the same for $\alpha_2$ and noting that $\sqrt{\abs{m_1}/n}(1-\abs{m}/n) \leq \sqrt{\abs{m_1}/n}(1-\abs{m_1}/n) \leq 2/(3 \sqrt{3})$ and noting that$160 + 4 \abs{m_1} \leq 160+4 \abs{m} = \abs{m}(160/\abs{m}+4) \leq 31 \abs{m}$ because $\abs{m} \geq 6$ gives the result in \eqref{eq:help.bound}.

To show the result in \eqref{eq:help.bound.m}, we use Lemma~\ref{lem:bound.rhohat.m} and Lemma~\ref{lem:bound.rho.m} together with the inequalities discussed above and the facts that $1+\mu \leq (1+\mu)^2$ and that $1-\abs{m}/n \leq 1$ to bound the sum in the second-to-last display from above by
\begin{align*}
22 & \exp \lrr{ - n \lrr{1 - \frac{\abs{m}}{n} }^5   \frac{ \alpha_1^2 \eps^2}{210 (1+\mu)^2(1+\eps)^2 } } \\
 & + (58 + 2 \abs{m_1})\exp \lrr{ - n \lrr{1 - \frac{\abs{m}}{n} }^5   \frac{ (1-\alpha_1)^2 \eps^2}{22534 (1+\mu)^2(1+ \eps)^2 } } \\
& + 22 \exp \lrr{ -n \lrr{1 - \frac{\abs{m}}{n} }^5   \frac{ \alpha_2^2 \eps^2}{840 (1+\mu)^2(1+\eps)^2 } } \\
& + (58 + 2 \abs{m_1}) \exp \lrr{ - n \lrr{1 - \frac{\abs{m}}{n} }^5   \frac{ (1-\alpha_2)^2 \eps^2}{19371(1+\mu)^2 (1+ \eps)^2 } }.
\end{align*}
Choosing $\alpha_1, \alpha_2 \in (0,1)$ such that the exponents in the preceding display coincide gives $\alpha_1=\sqrt{210}/(\sqrt{210}+\sqrt{22534})$ and $\alpha_2=\sqrt{840}/(\sqrt{840}+\sqrt{19371})$. Noting that $160 + 4 \abs{m_1} \leq 31 \abs{m}$ as before gives the result.

\end{proof}

\section{Technical details for Section \ref{sec:selection}}

\begin{proof}[Proof of Lemma~\ref{lem:bound.uniform}:]
By Bonferronis inequality the left-hand side in \eqref{eq:bound.uniform} is bounded from above by
\begin{align*}
\sum_{m \in \mathcal{M}_n} \P & \lrr{ \lrv{ \log \frac{\hat{\rho}^2(m)}{\rho^2(m)} } \geq \eps} \\
& \leq \sum_{m \in \mathcal{E}_n} 31 \abs{m} \exp \lrr{ - n \, \lrr{ \frac{\abs{m_1}}{n}}^2 \lrr{1-\frac{\abs{m}}{n} }^5 \frac{\eps^2}{14397(1+\eps)^2} }.
\end{align*}
where the inequality follows from \eqref{eq:help.bound} in Theorem~\ref{th:help.bound}. Replacing $\abs{m}$ by $s_n$ and $\abs{m_1}$ by $r_n$ in every summand gives an upper bound for the sum on the right-hand side in the preceding display. The statement of the Lemma follows now by replacing the sum by the number of summands.
Alternatively, we can use the upper bound in \eqref{eq:help.bound.m} in Theorem~\ref{th:help.bound} to bound the left-hand side in \eqref{eq:bound.uniform} from above by 
\begin{gather*}
\sum_{m \in \mathcal{M}_n} 31 \abs{m} \exp \lrr{ - n \, \ \lrr{1-\frac{\abs{m}}{n} }^5 \frac{\eps^2}{28279(1+\eps)^2(1+\mu(m))^2} },
\end{gather*}
where $\mu(m) = \theta'\Sigma(m) \theta/\sigma^2(m)$. We can upper bound the sum in the preceding display by replacing $\abs{m}$ by $s_n$ in every summand. Noting that $1+\mu(m) = (\sigma^2(m)+\theta'\Sigma(m)\theta)/\sigma^2(m) = \Var(y)/\sigma^2(m) \leq \Var(y)/\sigma^2 \leq d$ and replacing the sum by the number of summands gives the result. 

\end{proof}

\begin{proof}[Proof of Corollary~\ref{cor:bound.best}:]
Because $m_n^*$ is a minimizer of $\rho^2(\cdot)$, the left-hand side in \eqref{eq:true.perf} equals $\P (\log (\rho^2(\hat{m}_n^*)/ \rho^2(m_n^*) ) \geq \eps)$. On the event where $\hat{\rho}^2(m_n^*)>0$ as well as $\hat{\rho}^2(\hat{m}^*)>0$, which happens with probability 1, we have that
\begin{gather*}
\log \frac{\rho^2(\hat{m}_n^*)}{\rho^2(m_n^*)} = \log \frac{\rho^2(\hat{m}_n^*)}{\hat{\rho}^2(\hat{m}_n^*)} + \log \frac{\hat{\rho}^2(\hat{m}_n^*)}{\hat{\rho}^2(m_n^*)} + \log\frac{\hat{\rho}^2(m_n^*)}{\rho^2(m_n^*)}.
\end{gather*}
Because $\hat{m}_n^*$ is a minimizer of $\hat{\rho}^2(\cdot)$, the second term is nonpositive and we can bound the left-hand side in \eqref{eq:true.perf} from above by
\begin{align} 
 \P \lrr{ \log \frac{\hat{\rho}^2(\hat{m}_n^*)}{\rho^2(\hat{m}_n^*) } \leq - \eps/2 } +  \P \lrr{ \log \frac{\hat{\rho}^2( m_n^*)}{\rho^2(m_n^*) } \geq \eps/2 }. \label{eq:true.perf.1}
\end{align}
Using Bonferroni's inequality, we can bound the sum in \eqref{eq:true.perf.1} from above by
\begin{align*} 
& \sum_{m \in \mathcal{M}_n} \lrs{ \P \lrr{ \log  \frac{ \hat{\rho}^2(m)}{\rho^2(m) } \geq \eps/2 } +  \P \lrr{ \log \frac{ \hat{\rho}^2( m )}{ \rho^2(m) } \leq -\eps/2 }}.
\end{align*}
Using the same arguments as in the proof of the previous lemma shows the statements in \eqref{eq:true.perf} and \eqref{eq:true.perf.m}. The statements in \eqref{eq:est.perf} and \eqref{eq:est.perf.m} are direct consequences of Lemma~\ref{lem:bound.uniform}.
 \end{proof}

\section{Technical details for Section~\ref{sec:inference}}

\begin{proof}[Proof of Theorem~\ref{th:bound.tv}:]
Recall that $\mathbb{L}(m)$ denotes a Gaussian distribution with zero mean and variance equal to $\rho^2(m)$, i.e., $\mathcal{N}(0, \rho^2(m))$, and that $\widehat{\mathbb{L}}(m)$ denotes a Gaussian distribution with zero mean and variance equal to $\hat{\rho}^2(m)$, i.e., $\mathcal{N}(0, \hat{\rho}^2(m))$. Using the fact that $\Vert \mathcal{N}(0, \hat{\rho}^2(m)) - \mathcal{N}(0, \rho^2(m)) \Vert_{TV} = \Vert \mathcal{N}(0, \hat{\rho}^2(m)/\rho^2(m)) - \mathcal{N}(0, 1) \Vert_{TV}$ together with Lemma D.1 in \cite{leeb09}, we have
\begin{gather} \label{eq:bound.tv.1}
\Vert \widehat{\L}(m) - \L(m) \Vert_{TV} \leq \frac{\lrv{ \log(\hat{\rho}^2(m)/\rho^2(m))} }{\sqrt{2 \pi \exp(1)}} \leq \frac{ \lrv{ \log(\hat{\rho}^2(m)/\rho^2(m)) } }{4}.
\end{gather}
Hence, the statements follow from Theorem~\ref{th:help.bound} with $4 \eps$ replacing $\eps$.
\end{proof}

\begin{proof}[Proof of Corollary~\ref{cor:bound.tv.uniform}:]
Using the inequality in \eqref{eq:bound.tv.1}, we can bound the left-hand side in \eqref{eq:bound.tv.uniform} from above by
\begin{gather*}
\P \lrr{ \sup_{m \in \mathcal{M}_n} \Vert \widehat{\L}(m) - \L(m) \Vert_{TV} \geq \eps} \leq \P \lrr{  \sup_{m \in \mathcal{M}_n }\lrv{ \log \frac{\hat{\rho}^2(m)}{\rho^2(m)}} \geq 4 \eps }.
\end{gather*}
The result follows from Lemma~\ref{lem:bound.uniform} with $4 \eps$ replacing $\eps$.
\end{proof}

\begin{proof}[Proof of Corollary~\ref{cor:pi.valid}:]
Note that we have $y^{(0)} \in \mathcal{I}(\hat{m}_n^*)$ if and only if $y^{(0)} - \hat{y}^{(0)}(m) \in \lrs{ - Q_{(1-\alpha/2)} \hat{\rho}(\hat{m}_n^*),  Q_{(1-\alpha/2)} \hat{\rho}(\hat{m}_n^*) }$. Denote the interval in squared brackets by $A$ and let $\mathbb{L}(m, A)$ and $\mathbb{\widehat{L}}(m,A)$ denote the probability of $A$ under $\mathbb{L}(m)$ and $\mathbb{\widehat{L}}(m)$, respectively. Note that the prediction interval was chosen such that $\mathbb{\widehat{L}}(\hat{m}_n^*, A) = 1-\alpha$. Hence, the left-hand side of \eqref{eq:pi.valid} can be rewritten as
\begin{gather*}
\lrv{ \mathbb{\widehat{L}}(\hat{m}_n^*, A) - \mathbb{L}(\hat{m}_n^*, A)}.
\end{gather*}
But this term is clearly bounded from above by the total variation distance of $\mathbb{\widehat{L}}(\hat{m}_n^*)$ and $\mathbb{L}(\hat{m}_n^*)$ and the result follows from Corollary~\ref{cor:bound.tv.uniform}.
\end{proof}

\begin{proof}[Proof of Corollary~\ref{cor:pi.short}:]
Because $m_n^*$ and $\hat{m}_n^*$ are minimizers of $\rho^2(\cdot)$ and $\hat{\rho}^2( \cdot)$, respectively, we have the following inequality
\begin{gather*}
\log \frac{\hat{\rho}^2(\hat{m}_n^*)}{\rho^2(\hat{m}_n^*)} \leq \log \frac{\hat{\rho}^2(\hat{m}_n^*)}{\rho^2(m_n^*)} \leq \log \frac{\hat{\rho}^2(m_n^*)}{\rho^2(m_n^*)}.
\end{gather*}
We use the fact that $\log (\hat{\rho}(\hat{m}_n^*)/\rho(m_n^*)) =  \log (\hat{\rho}^2(\hat{m}_n^*)/\rho^2(m_n^*))/2$ together with the second inequality in the preceding display followed by Bonferroni's inequality to get
\begin{gather*}
\P \lrr{ \log \frac{\hat{\rho}(\hat{m}_n^*)}{\rho(m_n^*)} \geq \eps }  \leq \sum_{m \in \mathcal{M}_n } \P \lrr{ \log \frac{\hat{\rho}^2(m)}{\rho^2(m)} \geq 2 \eps }.
\end{gather*}
Similarly but using the first inequality in the second-to-last display, we get
\begin{gather*}
\P \lrr{ \log \frac{\hat{\rho}^2(\hat{m}_n^*)}{\rho^2(m_n^*)} \leq - \eps } \leq \sum_{m \in \mathcal{M}_n } \P \lrr{ \log \frac{\hat{\rho}^2(m)}{\rho^2(m)} \leq  - 2\eps }.
\end{gather*}
Taken together, we can bound the left-hand side in \eqref{eq:pi.short} from above by
\begin{gather*}
 \sum_{m \in \mathcal{M}_n } \P \lrr{ \lrv{ \log \frac{\hat{\rho}^2(m)}{\rho^2(m)} } \geq 2 \eps }
\end{gather*}
and the result follows from the same arguments as in the proof of Lemma~\ref{lem:bound.uniform} with $\eps$ replaced by $2 \eps$.
\end{proof}

\end{appendix}

\newpage

\bibliographystyle{abbrvnat}
\bibliography{/Users/ninahuber/Literatur/shorttitles,/Users/ninahuber/Literatur/Literatur}

\end{document}